\documentclass[12pt,reqno]{amsart}
\usepackage{fullpage}
\usepackage{amsfonts}
\usepackage{amssymb}
\usepackage{times}
\usepackage{graphicx}
\usepackage{amsmath,amssymb,bm,bbm}
\usepackage{amsmath}
\usepackage{mathtools}
\usepackage{enumerate}
\usepackage{hyperref}
\usepackage{epsfig}
\usepackage{textcomp}
 
\theoremstyle{plain}  
\newtheorem{thm}{Theorem}[section]
\newtheorem{cor}[thm]{Corollary}
\newtheorem{lem}[thm]{Lemma}
\newtheorem{prop}[thm]{Proposition}

\newtheorem{defn}[thm]{Definition}

\newtheorem{conj}[thm]{Conjecture}

\newcommand{\thmref}[1]{Theorem~\ref{#1}}

\theoremstyle{remark}
\newtheorem{rem}[thm]{Remark}

\numberwithin{equation}{section}

\def\f{\frac}

\def\({\left(}
\def \){ \right)}
\def\[{\left[}
\def \]{ \right]}
\def\Bl{\Bigl}
\def\Br{\Bigr}

\def\Ga{\Gamma}
\def\ta{\theta}
\def\al{{\alpha}}
\def\be{{\beta}}
\def\da{{\delta}}

\def\a{{\alpha}}
\def\b{{\beta}}

\def\ga{{\gamma}}

\def\t{{\theta}}
\def\ld{{\lambda}}
\def\l{{\lambda}}
\def\d{{\delta}}

\def\va{\varepsilon}

\def\NN{{\mathbb N}}

\def\RR{{\mathbb R}}
\def\SS{{\mathbb S}}

\def\supp{\operatorname{supp}}

\def\sph{\mathbb{S}^{d}}

\newcommand{\wh}{\widehat}

\begin{document}
\title[]{ISOTROPIC POSITIVE DEFINITE FUNCTIONS ON SPHERES}

\author{Han Feng}
\email{hanfeng@cityu.edu.hk}
\author{Yan Ge}
\email{yge3@ualberta.ca}

\date{Jan.01, 2022}

\begin{abstract}
	In this paper, we investigate the relationship between positive  definite functions on the unit sphere $\sph$ and on the Euclidean space $\RR^d$. For the dimension $d$ to be odd, a new technique is developed to establish the inheritance of positive (semi-)definite property from $\RR^d$ to $\sph$ and    the converse. For $d=2$, it is proved that  a function defined by
	$$f_{\t,\delta}(t)=(\t-t)_+^\delta, \quad \delta\geq \f{d+1}2 $$
	is positive  definite on the unit sphere $\mathbb{S}^2$ by restricting $\t$ in an absolute range. Our results can verify a conjecture proposed by  R.K. Beatson, W. zu Castell, Y. Xu   and a sharp P\'{o}lya type criterion for positive definite functions  on spheres.
\end{abstract}
\maketitle

\section{Introduction}
Given a metric space $(\Omega, \rho)$, an isotropic function $g$ is called positive definite (positive semi-definite) on $\Omega$ if for any integer $N\in \NN$ and
$\mathbf X_N=\{x_1,\ldots, x_N\}\subset \Omega$, the corresponding $N\times N$ symmetric matrix
$$g[\mathbf X_N]=\Bl\{g\Bigl(\rho(x_i,x_j)\Bigr)\Br\}_{i,j=1}^N$$
is positive definite (positive semi-definite), which means
$$\mathbf c^T g[\mathbf X_N] \mathbf c >0 (\geq 0), \quad \text{for all nonzero}\  \mathbf c\in \RR^N.$$
Positive definite functions play important roles in many applications, like scatter interpretation, kernel method and spatial statistics (see \cite{Buh}, \cite{BOOK}, \cite{FS}, \cite{HW}, \cite{Stewart}, \cite{Gnei2}).
The concern of construction proper positive definite functions, therefore, has been attracting interest.
In the past, such a problem has been extensively studied for the case $\Omega=\RR^d$, the Euclidean space, see \cite{Askey},\cite{Gnei2}. The remarkable Bochner's criterion \cite{BO1},\cite{BO2}, states that
a continuous function $g$ is isotropic positive definite on $\RR^d$ if and only if it can be the Fourier transform of a nonnegative finite-valued Borel measure on $\mathbb{R}^d$.
As a typical example, with the positivity of Bessel integral \cite{Gasper}, it can be proved that for any $\theta>0$, the function defined by\begin{equation}\label{trunct}	f_{\theta, \delta}(t):=(\theta-t)_{+}^{\delta}=\left\{\begin{array}{ll}		(\theta-t)^{\delta}, & t \leq  \theta \\		0, & t> \theta	\end{array}\right.\end{equation}is isotropic positive definite on $\RR^d$ when  {$\d\geq   \f {d+1} 2$}.

For $x=\(x_1,...,x_{d+1}\)\in \mathbb{R}^{d+1}$, we write $|x|$ as the norm. On unit spheres $\sph:=\{x \in \RR^{d+1}: |x|=1\}$, $d\geq 1$, Schoenberg  established a parallel result of Bochner criterion in his classical paper \cite{S2}. He  characterized that the isotropic continuous semi-positive definite functions on $\sph$ as those
functions of  the Gegenbauer expansion \begin{equation*} 
	g(\t)=\sum_{n=0}^{\infty}a_nC_n^{\l}(\cos\ta), \ \ \ta\in[0,\pi], \ \ \l:=\f{d-1}2
\end{equation*} in which all of the coefficients 	\begin{equation}\label{GegenbauerCoefficient}
	a_n:=\int_0^\pi g(\t)C^\l_n(\cos \t)\sin^{2\l}\t d\t\geq 0,\quad \forall \, n \in \mathbb{N}_0,
\end{equation}
and $\sum_{n=0}^{\infty}a_nC_n^{\l}(1)<\infty$.  
The   positive definite functions on $\sph$ came somewhat later in \cite{CMS}  and were characterized as functions with $a_n>0$ for infinitely many even and infinitely  many odd indices in  \eqref{GegenbauerCoefficient}.

Unfortunately, such criterion is often hard to verify when confronted with a particular example. The difficulty is to evaluate the integrals of \eqref{GegenbauerCoefficient}.
Instead, a P\'{o}lya criterion on the spheres was formulated due to Beatson, Castell and Xu \cite{BzX}, which provides a simple and efficient way to determine positive definite. Recently, Buhmann and J\"{a}ger \cite{BJ} established the  P\'{o}lya criterion on the spheres for conditional positive definite functions. However, the proof in \cite{BzX} relies on
a conjecture about the positive definite property of \eqref{trunct}, which can be equivalently stated as   the following  by adding the convergence of the series $\sum_{n=0}^{\infty}a_nC_n^{\l}(1)<\infty$.
\begin{conj}\label{conj1.9}
	Let  $\delta\geq \l+1$ and let $\lambda=\f{d-1}2 $.  Then for any $\t \in (0,\pi)$, the function
	$$f_{\t,\delta}(t)=(\t-t)_+^\delta$$    is isotropic positive definite on $\mathbb{S}^d$.  
\end{conj}

This conjecture was proved for $d=3,5,7$ in \cite{BzX}. Related recent work on this conjecture for the higher dimensional case can be found in the papers \cite{Xu}, \cite{NM}, \cite{LM}.
In this paper, we first consider the case when $d$ is odd and shall provide  a new technique to show that the positive definite property can be inherited from Euclidean space to the unit sphere, which will immediately verify the conjecture for all odd dimensions.  Our first main theorem states as follows.
\begin{thm}\label{thm1.3}
	Let $d$ be an odd integer $\geq 3$. Suppose that $g$ is a continuous function on $[0,\infty)$ with compact support on $[0,\pi]$. If $g$ is isotropic positive definite on $\RR^d$, then so it is on $\sph$.
\end{thm}
Some remarks are worthwhile to list.
\begin{rem} The condition of compact support is necessary. A counterexample can be $g(t)=\exp(-t^2)$, which is isotropic positive define on $\RR$ but not on $\mathbb S^1$.
\end{rem}

\begin{rem}
	The result was claimed as well by Nie and Ma \cite[Theorem 2.1]{NM}. However, we find it is difficult to follow the proof of Lemma 3.4 in \cite{NM} which plays a crucial role in the proof of the main theorem.	There exists one example sequence $2^kn+2^kj_0+2^{k-1}j_1+...+j_k$ where $k, n\in\NN$, $j_0,j_1,\ldots,j_k$ are bounded nonnegative integers, which can not be covered by Lemma 3.4 by taking $n=0$ and
	\begin{equation*}
		j_i=
		\begin{cases}
			0& {0\leq i\leq \f k2-1}\\
			1&  {\f k2\leq i\leq k}.
		\end{cases}
	\end{equation*}
	In this example,
	$$2^k\gamma_k+\gamma_0:=2^kn+2^kj_0+2^{k-1}j_1+...+j_k$$ 
	with $\gamma_0=0$ and
	$\gamma_k=n+j_0+2^{-1}j_1+\cdots+2^{-k}j_k=2^{-\f k2+1}-2^{-k}\to 0$ as $k\to \infty$, which implies that the conditions in Lemma 3.4(i) are not satisfied.  On the other hand, $$2^k\gamma_k+\gamma_0=2^k(2^{-\f k2+1}-2^{-k})=2^{k/2+1}-1\to \infty, \ \text{as}\ k\to \infty.$$
	Lemma 3.4(ii) can not be used since $2^kn+2^kj_0+2^{k-1}j_1+...+j_k$ is unbounded at this moment.
	Besides \cite{NM}, another related and similar result was obtained due to Ma and Lu \cite{LM}, in which the positivity of isotropic covariance matrices was proved to be preserved from Euclidean spaces to unit spheres when the dimensions are odd.
\end{rem}

\begin{rem}	
	In contrast with the proof in \cite{NM}, we shall first proceed our proof to functions with high regularity, which allows us to simplify our problem  by applying integration by parts. Then through inducing a proper identity approximation operator, additional regular assumptions will be avoided.
\end{rem}

\begin{cor}\label{cor1.5}
	Let $d$ be odd and let $g$ be an  isotropic   continuous  positive  semi-definite function on $\SS^d$ with   compact support on $[0,\pi]$. Then
	the function
	$$g(t) \( \f {\sin t}{t}\)^{d-1}$$
	is an  isotropic  positive  semi-definite function on both  $\RR^d$ and $\sph$.
\end{cor}

The converse of \thmref{thm1.3} will  be considered for the  positive semi-definite functions.
In this situation, the restriction on dimensions can be removed. Precisely, we state the following theorem.
\begin{thm}\label{thm1.4}
	Suppose that $g$ is a continuous function on $[0,\infty)$ with compact support on $[0,\pi]$. Let $d> 1$. If $g(\t^{-1}\cdot )$ is isotropic  positive  semi-definite on $\sph$ for $\t\in(0,1)$, then $g$ is isotropic positive  semi-definite on $\RR^d$.
\end{thm}

\begin{rem}	 In the published version, the statement of Theorem~\ref{thm1.4} should be corrected to require that
$g(\theta^{-1}\cdot)$ is isotropic positive semi-definite on $\mathbb{S}^d$ for all $\theta\in(0,1)$, as stated above. The proof in the published paper proceeds under this condition and therefore remains valid.
\end{rem}

Next, we consider the conjecture in the even dimensional case.
In this situation, the
approach of Theorem \ref{thm1.3} does not seem to be extendable due to
the fact that there is no such an extensional relation Lemma \ref{lem2.3} in even dimensional cases.
Our work will focus
  on the case $d=2$, and a restriction on the parameter $\t$.  Then we shall partly prove the conjecture.  
Our main result in this case can be stated as follow:
\begin{thm}
	Let $d=2$ and let $\delta\geq\f {d+1}2$. There exists a  constant $c$ such that the function
	$$f_{\t,\delta}(t)=(\t-t)_+^\d$$    is isotropic positive definite on $\mathbb{S}^2$ whenever $0\leq\t\leq c$,
	where $c$ is an absolute constant $0 < c < \pi$.
 
\end{thm} To the best of our kownledge, there is no  result for  even dimensional cases.  We will present a detailed discussion of the proof in Section   \ref{section5}.

The paper is organized as follows:  Section 2 contains preliminary results. We will prove an extensional relation about the Jacobi polynomials in Lemma \ref{lem2.3} and present a  known asymptotic relationship between Jacobi polynomial and  Bessel function in Lemma \ref{lem3.1}. They are most often used in our proof. In Section 3, we present the proof of our main result Theorem \ref{thm1.3}, which can be done in three steps. In particular, in step 2, we divide the estimation into three cases. The proof of the  Corollary \ref{cor1.5} and Theorem \ref{thm1.4} are given in Section 4. Section 5 is devoted to the problem of Conjecture \ref{conj1.9}. We will give the proof  when   $d=2$ and $\t$ is small enough. The proof is rather involved, and we break it into several steps. In Appendix A and B, we will provide some useful formulas in the proof of  Section 5.

 \section{Preliminaries and Lemmas}
 
 In this section, we first prove an extensional relation Lemma \ref{lem2.3} on the Gegenbauer polynomials that will be useful in the proof of Theorem \ref{thm1.3}. We also   provide the asymptotic relation formula between Jacobi polynomials and Bessel functions that  we will use   several times in the sequel. Finally, we will recall the  fact that the Fourier transform of  isotropic functions can be represented by Bessel functions.
 
 Let us start with some  notations. We use the standard notation of $P_n^{(\a,\b)}(x)$  for the Jacobi polynomials  where parameters $\a,\b\in\mathbb{R}$ and $n\in\mathbb{N}$, and use   $C_n^\l(x)$ for the Gegenbauer polynomials of degree $n$. The Gegenbauer polynomial  is a special case of the Jacobi  polynomial with the well-known relation $$C_n^\l(x)= \frac{\Gamma(\l+\f12)\Gamma(n+2\l)}{\Gamma(2\l)\Gamma(n+\l+\f12)}P_n^{(\l-\f12,\l-\f12)}(x), \ \ \  \l>-\f12.$$  For later applications, we denote $R_n^\ld(x):= \f {C_n^\ld(x)}{C_n^\ld (1)}   $ as  normalized Gegenbauer polynomials.

We first recall the following two results on Jacobi polynomials.
 \begin{lem}The following formula holds: \cite[(4.5.4)]{Sz}
  \begin{equation*}\label{eqn 5.6}
   P_n^{(\a ,\b+1)}(x)=\f{2}{2n+\a+\b+2}\frac{(n+\b+1)P_n^{(\a ,\b)}(x)+(n+1) P_{n+1}^{(\a ,\b)}(x)}{1+x}.
  \end{equation*}
 \end{lem}

  \begin{lem}The following formula holds: \cite[(4.1.5)]{Sz}
 	\begin{equation*}\label{lem2.2}
 	P_{2n}^{(\a ,\a)}(x)=\f{\Gamma(2n+\a+1)\Gamma(n+1)}{\Gamma(n+\a+1)\Gamma(2n+1)}  P_{n }^{(\a ,-\f12)}(2x^2-1).
 	\end{equation*}
 \end{lem}

By using the above two  formulas, we can establish an explicit expression for $R_n^\ld (\cos t) (\cos \f t2)^{2\ld}$. That is, we can write $R_n^\ld (\cos t) (\cos \f t2)^{2\ld}$  in terms of a linear combination of the terms $R_{2j}^{\ld}(\cos \f t2)$, with $j=n,...,n+\l$, and more importantly, all the coefficients are nonnegative.  This relation plays an
important role in our proof.

\begin{lem}\label{lem2.3} If $\ld$ is a positive integer, then  there exists a sequence $\{a_{n,j}^\ld \}_{j=n}^{2n+2\ld}$ of positive numbers such that  for any $t \in [0,\pi]$,
	\begin{equation}\label{4.2}
	R_n^\ld (\cos t) (\cos \f t2)^{2\ld} =\sum_{j=n}^{n+\ld} a_{n,j} ^\ld R_{2j}^{\ld}(\cos \f t2).
	\end{equation}
\end{lem}
\begin{proof}
	First, applying the Lemma \ref{eqn 5.6}, substituting $x$ by $\cos(\f t2)$, $\a$ by $\a-\f 12$, and $\b$ by $\b-\f 12$, we get for $t\in [0,\pi]$
	\begin{align}\label{4.3}
		 (\cos \f t2)^2 P_n^{(\al-\f12, \be+\f12)}(\cos t)  
  = A_n^{\al,\be} P_n^{(\al-\f12, \be-\f12)}(\cos t ) + B_n^{\al,\be} P_{n+1}^{(\al-\f12, \be-\f12)}(\cos t), \end{align}  where
	$$ A_n^{\al,\be} =\f {n+\be+\f12} {2n+\al+\be+1},\    \  B_n^{\al,\be} =\f {n+1} {2n+\al+\be+1}.$$
	
	Thus,
	using \eqref{4.3} $\ld$ times, we obtain
	\begin{align}\label{eqn2.5}
	C_n^{\ld} (\cos t) (\cos \f t2)^{2\ld}
	=\sum_{j=n}^{n+\ld} \ga_{j} P_j^{(\ld-\f12, -\f12)}(\cos t),
	\end{align}
	where for all  
	$n\leq j\leq n+\lambda$, $\gamma_{j}$ are some nonnegative coefficients.	Next, to complete the proof, we will use the Lemma \ref{lem2.2} with substituting $x$ by $\cos(\f t2)$, $\a$ by $\a-\f 12$, and get for $t \in [0,\pi]$ and $\al>-\f12$,
	\begin{equation}\label{eqn2.6}
		P_n^{(\al-\f12, -\f12)}(\cos t) = \frac{\Gamma(n+\a+\f12)\Gamma(2n+1)}{\Gamma(2n+\a+\f12)\Gamma(n+1)}  P_{2n}^{(\al-\f12, \al-\f12)} (\cos \f t2).\end{equation}
	Combining the equations \eqref{eqn2.5} and \eqref{eqn2.6},   we obtain the desired result \eqref{4.2}.
\end{proof}
 
  The following well-known relation connecting Jacobi polynomial and   Bessel function is crucial for our proof (See \cite{Sz}, \cite{FW}).
  \begin{prop}\label{lem3.1} 
  	For $\al>-\f12$, we have  for  $t\in (0,\pi)$,
  	\begin{align*} 
  		\f{P_n^{(\al,\al)}(\cos t)}{P_n^{(\al,\al)}(1)} & =2^\al\Ga(\a+1) \Bl(\f t {\sin t}\Br)^{\al+\f12}\Bl[  j_\a(Nt)+  O(n^{-1})\Br],
  	\end{align*}
  	where $N=n+\a+\f12$,    $j_\al (z)=z^{-\al} J_\al(z)$, and here and in what follows $J_\al(z)$ is the Bessel function of the first kind, $$J_\al(z)=\sum_{v=0}^\infty\frac{(-1)^v(\f z2)^{\a+2v}}{v!{\Gamma(v+\a+1)}}.$$ The $O$-term is uniform with respect to $t\in [0, \pi-\va]$, $\va$ being an arbitrary positive number.
  \end{prop} In particular, when $\a=\b=0$, \cite[(4.37), (4.38)]{FW} provides a useful consequence in Legendre polynomial case: for $t\in [0,\f\pi 2]$,

  \begin{align}\label{C.1}	\Bl|{P_n}(\cos t)-\(\f t{\sin t}\)^{\f12}J_0(Nt)\Br|  \leq\f{0.1711}n.  \end{align}

  Finally, let us  recall the Fourier transform $\widehat{f_d}$  of an isotropic  function ${f_d}$ defined on $\mathbb{R}^d$. As we have known in \cite{SW}, it
  turns out to be isotropic  as well. An important fact is that the Fourier transform of isotropic functions can be represented  in terms of the Bessel functions.
  \begin{prop} \label{prop2.5}  
 	For the   isotropic function $f_d\in L^1(\mathbb{R}^d)$,   there exists a one dimension function $f$ defined on $[0,\infty)$ such that ${f_d}(x)=  {f}(|x|)$ for   $x\in \mathbb{R}^d$,  and
 	\begin{align*} 
 		\widehat{f_d}(\xi)=(2\pi)^{\f d2}\int_0^{\infty}f(u) j_{\f{d-2}2} (|\xi|u) u^{d-1} du, \  \ \ \ \forall \xi\in\RR^d.
 	\end{align*}
 
 \end{prop}

\section{Proof of Theorem \ref{thm1.3}}
We are now in a position to prove Theorem \ref{thm1.3}, that is,  for an odd integer $d$ and a continuous function $g$ on $[0,\infty)$ with compact support  on $[0,\pi]$, if $g$ is isotropic positive definite on $\RR^d$, so it is on $\sph$. According to  Schoenberg's theorem, we  need to prove that the Gegenbauer coefficients of $g$
\begin{align}\label{claim3.1}
	a_n&=\int_0^\pi g(t)C^\l_n(\cos t)\sin^{2\l}t dt > 0,\quad \forall \, n \in \mathbb{N}_0,
\end{align}
and 
\begin{align}\label{claim3.2}
\sum_{n=0}^{\infty}a_nC_n^{\l}(1)<\infty, \ \ \  where  \ \ \lambda=\frac{d-1}2.
\end{align}

 We first verify the validity of $\eqref{claim3.2}$. The proof follows a standard argument  in \cite{S2}. The series $\sum_{n=0}^{\infty}a_nC_n^\l (\cos t)$ is Abel-summable for every $t\in [0,\pi]$. Hence, for $t=0$, $$\sum_{n=0}^{N}a_n C_n^\l(1)\leq \lim\limits_{r\rightarrow 1^-}\sum_{n=0}^{\infty}a_nC_n^\l (1)r^n<\infty.$$ Thus we have
 $\sum_{n=0}^{\infty}a_nC_n^\l (1)$ converges.\\

We now embark on proving equation  \eqref{claim3.1}, which will be carried out in three steps.

  \subsection{Step 1: Reduction}\label{reduction}
  Let us begin with a series of reductions. In order to prove equation \eqref{claim3.1}, we claim that it is enough to prove a slightly stronger statement: for any $n\in\mathbb{N}_0,$ and $\t\in(0,\pi]$, 
\begin{equation}\label{3.3}
\int_0^{\t} g(\f{t\pi}{\t}) C_n^\ld (\cos t) (\sin t)^{2\ld}\, dt > 0,
\end{equation} 
and it is obvious to see \eqref{claim3.1} is exactly a particular case when $\t=\pi$. We now continue to deal with the claim \eqref{3.3}. Without loss of generality, in \eqref{3.3},  we may replace $f(\f{t}\ta)$ with $g(\f{t\pi}\ta)$, i.e. $f(t)=g(\pi t)$, where $f$ is isotropic continuous positive definite on $\RR^d$ with compact support on $[0,1]$. Hence, showing the claim \eqref{3.3}  is equivalent to show for any $n\in\mathbb{N}_0,$  $\t\in(0,\pi]$,
\begin{equation*}\label{3.4}
 \int_0^\ta  f \Bl(\f t \ta\Br)  C_n^\ld(\cos t) (\sin t)^{2\ld} \, dt >0.\
 \end{equation*}
 For simplicity in the proof, we will deal with the following integral and verify its positivity:

 \begin{equation}\label{3.5}
I_n(\ta):= \ta^{-2\ld-1}\int_0^\ta  f \Bl(\f t \ta\Br)  R_n^\ld(\cos t) (\sin t)^{2\ld} \, dt >0,\
\end{equation}
where $\t\in(0,\pi]$, and $n\in\mathbb{N}_0$.

Thus, by the arguments above, the previous claim \eqref{claim3.1} reduces to show $I_n(\ta)>0$ for all $n\in\mathbb{N}_0$ and every $\ta\in (0, \pi]$. Furthermore, with the help of Lemma \ref{lem2.3}, we would restrict our analysis on the integral \eqref{3.5} on a small range of $\t$. More precisely, we have if $I_n(\ta) >0$, where $\ta\in (0, \va_0)$, and $\va_0$ is a sufficiently small positive parameter, then $I_n(\ta) >0$ for every $\t\in(0,\pi]$. Indeed, this claim follows from the fact that for every $\t\in(0,\pi]$, there exists a $\ta_0\in (0, \va_0)$, where $\va_0$ is a sufficiently small positive parameter, such that   $I_n(\t)$ can be written as a sum of $I_n(\t_0)$ with positive coefficients,  which establishes the desired claim.

Thus, given the whole reduction arguments above, the proof of the claim \eqref{claim3.1} is finally reduced to show the following Lemma:
  \begin{lem}\label{lem3-3}
		 Let $d$ be an odd integer $\geq 3$ and let $f\in C[0,\infty)$ with $\supp( f) \subset [0, 1]$. If $f$ is an isotropic positive definite function on $\RR^d$, then
		\begin{equation}\label{3.6}
	I_n(\ta)= \ta^{-2\ld-1}\int_0^\ta  f \Bl(\f t \ta\Br)  R_n^\ld(\cos t) (\sin t)^{2\ld} \, dt>0\
	\end{equation}
	for all $n\in\mathbb{N}_0$ and every $\ta\in (0, \va_0)$, where $\va_0\in (0,1)$ is a sufficiently small positive parameter.
\end{lem}
Now, we will  give the proof of the Lemma \ref{lem3-3} in the following step 2 and 3.

\subsection{Step 2:}
In this step, we shall prove Lemma \ref{lem3-3} for the functions  with high regularity, that is, under the additional assumptions  that $f\in C^{\ld+3} [0,\infty)$, and $f'(0) <0$. We will divide the proof of \eqref{3.6} into the following three cases: (i) $n\ta \ge A$;  (ii) $n\ta\leq B$; (iii)  $B \leq n\ta \leq  A$, where $A>1$ and $B\in (0,1)$ are certain parameters depending only on $f$.

\subsubsection{Case (i).}

In this case, we shall prove that there exists a constant $A>1$ depending only on $f$ such that \eqref{3.6} holds whenever $n\ta \ge A$.

First, we need a definition.
\begin{defn}\label{def3.2}
	Let  $\ld$ be a positive integer. For $j=1,\cdots,\ld$, we define $$F_{0}(t)= f(\ta^{-1} t)(\sin t)^{2\ld} \ \ and\ \
	 F_{j} (t)=\Bl( \f {F_{j-1} (t)}{\sin t}\Br)'.$$
\end{defn} Clearly,  for $0\leq j\leq \ld$,  $F_{j}\in C^{\ld-j+3}[0,\infty)$,  $F_{j}(t)=0$    for $t\ge \ta$, and $ F_{j}(t)=O(t^{2(\ld-j)})$ as  $t\to 0^+$. The functions $F_j(t), j=1,2,...,\ld,$ have the following decomposition.
\begin{lem}\label{lem-6-4}
	For $0\leq t\leq \t$ and $\ell=1, 2, \cdots, \ld$,
\begin{align}\label{3.7}
F_\ell (t) = \sum_{j=0}^\ell \ta^{-j} f^{(j)} (\ta^{-1} t) \sum_{k=0}^{[\f {\ell-j}2]} \al_{\ell, k}^{(j)} (\sin t)^{2\ld-2\ell+j+2k} (\cos   t)^{\ell-j-2k},
\end{align} where the $\al_{\ell, k}^{(j)}$ are constants,
\begin{align}
\al_{\ell, 0}^{(0)} &=(2\ld-1)(2\ld-3)\cdots (2\ld-2\ell+1),\label{3.8}\\
\al_{\ell+1,0}^{(1)}&=\al_{\ell,0}^{(0)} +2(\ld-\ell) \al_{\ell,0} ^{(1)},\   \  \al_{1,0}^{(1)} =1.\label{3.9}
\end{align}  
In particular, \begin{align}\label{eqn3.9}
F_\ld (t)&=\Bl[ (2\ld-1)!! \cos^\ld t +  \sum_{k=1}^{[\f {\ld}2]} \al_{\ld, k}^{(0)} (\sin t)^{2k} (\cos \ t)^{\ld-2k}  \Br]  f (\ta^{-1} t)\notag\\
&+ \sum_{j=1}^\ld \ta^{-j} f^{(j)} (\ta^{-1} t) \sum_{k=0}^{[\f {\ld-j}2]} \al_{\ld, k}^{(j)} (\sin t)^{j+2k} (\cos t)^{\ld-j-2k}.
\end{align}

\end{lem}

\begin{proof}
 The proof uses induction by $\ell$.
When $\ell=1$, we have $$F_1(t)=\(\frac{F_0(t)}{\sin t}\)'=\t^{-1}f'(\f t\t)(\sin t)^{2\l-1}+(2\l-1)f(\f t\t)(\sin t)^{2\l-2}\cos t=\text{the right-hand side of \eqref{3.7}}$$ and $\al_{1, 0}^{(0)}=(2\l-1)$, $\al_{1, 0}^{(1)}=1.$

Suppose that the statement holds for
$\ell$.  Applying  the definition \ref{def3.2}, we have
\begin{align}
\begin{split}\label{3.11}
F_{\ell+1} (t)  =&\(\frac{F_\ell(t)}{\sin t}\)'= \sum_{j=0}^\ell  \sum_{k=0}^{[\f {\ell-j}2]}  \Big(  \ta^{-j-1} f^{(j+1)} (\frac{t}{\t}) \al_{\ell, k}^{(j)} (\sin t)^{2\ld-2\ell+j+2k-1} (\cos t)^{\ell-j-2k}\\
&+\ta^{-j} f^{(j)} (\frac{t}{\t}) \al_{\ell, k}^{(j)} ({2\ld-2\ell+j+2k-1})(\sin t)^{2\ld-2\ell+j+2k-2} (\cos  t)^{\ell-j-2k+1}\\
&+\ta^{-j} f^{(j)} (\frac{t}{\t}) \al_{\ell, k}^{(j)} ({j+2k-\ell})(\sin t)^{2\ld-2\ell+j+2k} (\cos  t)^{\ell-j-2k-1} \Big)\\
=&\sum_{j=0}^{\ell-1}  \sum_{k=0}^{[\f {\ell-j}2]}  \Big(   \ta^{-j-1} f^{(j+1)} (\frac{t}{\t}) \al_{\ell, k}^{(j)} (\sin t)^{2\ld-2\ell+j+2k-1} (\cos t)^{\ell-j-2k} \Big)\\
&+ \ta^{-\ell-1} f^{(\ell+1)} (\frac{t}{\t}) \al_{\ell, 0}^{(\ell)} (\sin t)^{2\ld-\ell-1}\\
&+\sum_{j=0}^{\ell}  \sum_{k=0}^{[\f {\ell-j}2]} \Big( \ta^{-j} f^{(j)} (\frac{t}{\t}) \al_{\ell, k}^{(j)} ({2\ld-2\ell+j+2k-1})(\sin t)^{2\ld-2\ell+j+2k-2} (\cos  t)^{\ell-j-2k+1} \Big)\\
&+\sum_{j=0}^{\ell}  \sum_{k=0}^{[\f {\ell-j}2]} \Big( \ta^{-j} f^{(j)} (\frac{t}{\t}) \al_{\ell, k}^{(j)} ({j+2k-\ell})(\sin  t)^{2\ld-2\ell+j+2k} (\cos t)^{\ell-j-2k-1} \Big). \\
\end{split}
\end{align}

Observing the power of each term in \eqref{3.11}, we can exactly rewrite $F_{\ell+1} (t)$ by  a double sum in terms of $\t^{-j}f^{(j)} (\frac{t}{\t})$, $(\sin t)^{2\ld-2\ell+j+2k-2}$, and $(\cos t)^{\ell-j-2k+1}$ as $j$ runs from $0$ to $\ell+1$ and $k$ runs from $0$ to $[\f {\ell+1-j}2]$. It remains to prove the coefficient satisfies \eqref{3.8} and \eqref{3.9}, that is, we need to verify that
\begin{align}
\al_{\ell+1, 0}^{(0)} &=(2\ld-1)(2\ld-3)\cdots (2\ld-2\ell+1)(2\ld-2\ell-1),\label{3.12}\\
\al_{\ell+1,0}^{(1)}&=\al_{\ell,0}^{(0)} +2(\ld-\ell) \al_{\ell,0} ^{(1)},\   \  \al_{1,0}^{(1)} =1.\label{3.13}
\end{align}
To show \eqref{3.12}, we need to consider the coefficient of   term  $f(\frac{t}{\t})(\sin t)^{2\ld-2\ell-2} (\cos t)^{\ell}$ in \eqref{3.11}, from which we have $$\al_{\ell+1, 0}^{(0)}=\al_{\ell, 0}^{(0)}(2\l-2\ell-1) =(2\ld-1)(2\ld-3)\cdots (2\ld-2\ell+1)(2\ld-2\ell-1).$$\\
To show \eqref{3.13}, we need to consider the coefficient of   term  $\t^{-1}f'(\frac{t}{\t})(\sin t)^{2\ld-2\ell-1} (\cos t)^{\ell}$ in \eqref{3.11}, from which we have $$\al_{\ell+1, 0}^{(1)}=\al_{\ell, 0}^{(0)}+\al_{\ell, 0}^{(1)}(2\l-2\ell).$$
This proves that the identity \eqref{3.7} holds in the case of $\ell+1$ and completes the induction. In particular, if $\ell=\ld$, then \eqref{eqn3.9} holds.
\end{proof}

Recall the following well-known properties on Gegenbauer polynomials:
\begin{align}\label{5.5}
\Bl( R_{n+1}^{\mu-1} (x)\Br)' =\f {(n+1)(n+2\mu-1)}{2\mu-1} R_n^\mu (x),\   \  \
\end{align}
\begin{align}\label{4-2-1}
\lim_{\mu\to 0+} R_n^\mu(\cos t)= \cos (n t).
\end{align}
Using   \eqref{5.5},  \eqref{4-2-1} and integration by parts $\ld$ times on \eqref{3.6}, we obtain
\begin{align}\label{5.7}
\ta^{2\ld+1} I_n(\ta)=c_n(\ld)\int_0^\t F_\ld (t) \cos ((n+\ld)t)\, dt
\end{align}
for some constant $c_n(\ld)>0$. Note that \eqref{3.8} and \eqref{3.9} imply that
$$ \al_{\ell+1, 0}^{(1)} +\al_{\ell+1, 0}^{(0)} =2(\ld-\ell) \Bl(\al_{\ell, 0}^{(1)} +\al_{\ell,0}^{(0)}\Br),$$
which in turn implies that
$$\al_{\ld, 0}^{(1)}+\al_{\ld,0}^{(0)}=2^{\ld}\ld!=(d-1)!!.$$ 
Thus, by our assumption on the function $f$, it is easily seen that   $F_\ld$ has the following properties:\begin{enumerate}[\rm (i)]
	\item  $F_\ld\in C^{3}[0,\infty)$ and  $F_\ld$ is supported on $[0,\ta]$;
	\item $F_\ld '(0) =\Bl(\al_{\ld,0}^{(0)} +\al_{\ld,0}^{(1)}\Br) \ta^{-1} f'(0)= (d-1)!! \ta^{-1} f'(0)<0,$ and
	$|F_\ld'''(t)|\leq C_f \ta^{-3}$ for any $t\in [0,\ta]$.
\end{enumerate}

To this end,  setting  $N=n+\ld$, and  using integration by parts three times on the right-hand side of \eqref{5.7}, we obtain
\begin{align*}
\int_0^\t F_\ld (t)\cos (Nt)\, dt
&=-\f 1 {N^2} F_\ld'(0) +\f1{N^3}\int_0^\t F_\ld'''(t) \sin (Nt)\, dt\\
&\ge N^{-2} \ta^{-1} \Bl[ -(d-1)!! f'(0) - C \f 1{N\ta}\Br],
\end{align*}
which is positive provided that $n\ta \ge A$ and $A$ is sufficiently large   depending only on $f$.
This proves \eqref{3.6} in the case of $n\ta \ge A$.

\subsubsection{ Case (ii).}
In this case, we shall prove that there exists a constant $B\in (0,1)$ depending only on $f$ such that \eqref{3.6} is true whenever $n\ta\leq B$.

To see this, we rewrite $I_n(\ta)$ into two parts
\begin{align*}
	I_n(\ta)& =\ta^{-2\ld-1}\Bl[\int_0^\infty f(\ta^{-1} t)  t^{2\ld}\, dt + \int_0^{\ta} f(\ta^{-1} t)  \Bl( R_n^{\ld} (\cos t) \f {\sin^{2\ld} t} {t^{2\ld}} -1\Br)t^{2\ld}\, dt\Br].
\end{align*}
Note that for any $t\in [0,\t]$,
\begin{align*}
	|R_n^\l(\cos t)(\f{\sin t}{t})^{2\l}-1|&\leq|R_n^\l(\cos t)(\f{\sin t}{t})^{2\l}-(\f{\sin t}{t})^{2\l}|+|(\f{\sin t}{t})^{2\l}-1|\\
	&\leq|(\f{\sin t}{t})^{2\l}|\cdot|R_n^\l(\cos t)-R_n^\l(\cos 0)|+C_dt\\ 
	&\leq nt||R_n^\l(\cos \cdot)||_{\infty}+C_dnt\leq C'_d nt,
\end{align*}
where $C_d$ and $C'_d$ are  constants depending on $d$, and the second last inequality follows from Bernstein's inequality for trigonometric polynomials.
Then using the same notation in Proposition~\ref{prop2.5} and noticing that $ \wh{f}_d (0)>0$,
\begin{align*}
	I_n(\ta) &\ge \int_0^\infty f(t) t^{d-1}\, dt - C'_d n \ta^{-2\ld-1} \int_0^\ta t^{2\ld+1}\, dt\\
	&\ge c_d \wh{f}_d (0) -C'_d n,
\end{align*}
which implies that $I_n(\ta)>0$ whenever $n\ta\leq c_d \wh{f}_d (0)/C'_d$.

\subsubsection{Case (iii).}\label{case3}

In this case, we shall prove \eqref{3.6} for the remaining case $B \leq n\ta \leq A$.

Let $N=n+\ld$, and by substitution $t=\f{t'}{N}$, we can write
$$I_n(\ta) = N^{-1} \ta^{-2\ld-1} \int_0^{N\ta} f\Bl( \f {t'} {N\ta}\Br) R_n^\ld (\cos\f {t'} N) \Bl(\sin \f {t'} N\Br)^{2\ld}\, d{t'}.$$
Using Proposition ~\ref{lem3.1}, we have that for $0\leq {t'} \leq N\ta \leq A$,
$$ R_n^\ld (\cos \f {t'} N) =c_\ld \Bl(\f {{t'}/N} { \sin({t'}/N)}\Br)^{\ld} \Bl [ j_{\ld-\f12} ({t'}) +O(n^{-1}) \Br],$$
where $c_\ld =2^{\ld-\f12} \Ga(\ld+\f12)$. Set $O_{A,B}$ term to be uniform in  $n$, ${t'}$, $\ta$, but may depend on the constants $A$ and $B$.  By  substitution $x=\frac{t'}{N\t}$, we have
\begin{align*}
I_n(\ta)
&=c_\ld \int_0^1 f(x) j_{\ld-\f12} (N\ta x) x^{2\ld} \Bl( \f {\sin (\ta x)}{\ta x} \Br)^{\ld} \, dx +  O_{A,B}(n^{-1})\\
&\ge c_\ld'  \int_0^1 f(x) j_{\ld-\f12} (N\ta x) x^{2\ld}\, dx + O_{A,B}(n^{-1})\\
&\ge c_\ld' \min_{B \leq |\xi|:=N\t\leq A } \int_0^\infty f(x) j_{\ld-\f12} (|\xi|x) x^{2\ld}\, dx + B^{-1}\ta^{ } O_{A,B}(1)\\
&=c_\ld''  \min_{B \leq |\xi|\leq A }\wh{f}_d (\xi) -C_{A,B} B^{-1} \t ,
\end{align*}
which is positive if $\t$ is small enough. The last step using the same notation in Proposition \ref{prop2.5}.  This shows \eqref{3.6} in this case.

\subsection{Step 3:}
In this step, we will show Lemma \ref{lem3-3} holds for a general continuous positive definite function  $f$ defined on $\RR^d$  without high regularity. The proof is based on the previous step 2. The main task in this proof is using the function $f$ to construct a  function $f^\va$ with high regularity, such that the Gegenbauer coefficients of $f^\va$ can approximate the Gegenbauer coefficients of $f$. The construction
requires two ingredients.

The first is a  finite  Borel measure defined on $\RR^d$:
$$ \mu_f:=f_d\cdot \chi_{B_{1-\va}} + M_{f_\va} m_\va \da_0,$$ where    $f_d$ to be the isotropic function on $\RR^d$ given by $f_d(x)=f(|x|)$ for $x\in\RR^d$; $B_r:=\{x\in \RR^d: |x|\leq r\}$ is the ball of radius $r$ on $\RR^d$; For $\va\in (0,1)$, $M_{f_\va} : =\max_{t\in [1-\va, 1]} |f(t)|$,
 $\da_0$ denotes the Dirac measure supported at the origin, and $$m_\va:=B_1-B_{1-\va}=\Bl|\{ x\in\RR^d:\  \  1-\va \leq |x|\leq 1\}\Br|=d^{-1} |\sph| \Bl( 1-(1-\va)^d\Br)\leq C \va.$$ Note that
\begin{align*}
\wh{\mu_f} (\xi) &=\int_{|x|\leq 1-\va} f(|x|) e^{-2\pi i x\cdot \xi}\, dx + m_\va M_{f_\va} \\
&=\wh{f_d} (\xi) +m_\va M_{f_\va} -\int_{1-\va<|x|\leq 1} f(|x|) \cos(2\pi i x\cdot \xi) dx\\
&\ge \wh{f_d} (\xi)+\int_{1-\va<|x|\leq 1} (M_{f_\va}-f(|x|)  )dx>0.
\end{align*}

The second ingredient is  an approximate identity:
$$\phi(x): =a_\ld (1-|x|)_+^{\ld+6},\   \  x\in\RR^d,$$
where the constant $a_\ld>0$ is chosen so that
$\int_{\RR^d} \phi(x)\, dx=1$.
It has been known that  $\wh{\phi}(\xi)>0$ for all $\xi\in\RR^d$  (See \cite{Gasper}).

Let $x=\f t\t$. Note that when $|t|<\t$, we have $0<|x|=|\f t\t|<1$. Define $f^\va_d(\f t\t):=f_d^\va(x)$ and 
$$ f_d^\va (x) :=\phi_{\va} \ast {\mu_f}(x) =\int_{\RR^d} \phi_\va(x-y) \, d\mu_f(y)=f_d \cdot \chi_{B_{1-\va}} \ast \phi_{\va} (x) + m_\va M_{f_\va} \phi_{\va} (x),$$
where   $\phi_{\va} (x) =\va^{-d} \phi(x/\va) $. Clearly, $f_d^\va(x)=f^\va(|x|)$ is an isotropic function on $\RR^d$, and
$$\wh {f_d^\va} (\xi) =\wh{\phi}(\va \xi) \wh{\mu_f} (\xi) >0,\   \ \xi\in \RR^d.$$ Since    $\phi_{\va} \in C_c^{\ld+5} (\RR^d)$  is supported on $\{x\in\RR^d:\  \ |x|\leq \va\}$,
and since $f_d \cdot \chi_{B_{1-\va}} \in L^1(\RR^d)$ is supported on $\{x\in\RR^d:\   \  |x|\leq 1-\va\}$, it follows that
$f^\va$ is a $(\ld+5)$-times continuously differentiable function on $[0,\infty)$
that is  supported on $[0, 1]$,
In addition,
$$( f^\va)'(0) =m_\va  M_{f_\va} \phi_{\va} '(0)= -\va^{-d-1}m_\va (f_d \cdot \chi_{B_{1-\va}}) a_\ld (\ld+6) <0.$$ Applying the conclusion that has already been proven in step 2, we obtain   that for any $\va\in (0,1)$ and $n\in\mathbb{N}_0$,
\begin{equation*} 
\int_{0}^\t f^\va (\f t\t) R_n^\ld (\cos t) (\sin t)^{2\ld}\, dt >0.
\end{equation*}
On the other hand, we have  
\begin{align*}
 \Bl|\int_0^\t \(f^\va(\f t\t) -f(\f t\t)\) R_n^\ld (\cos t) (\sin t)^{2\ld}\, dt \Br|&\leq \int_0^\t | f^\va(\f t\t) -f(\f t\t)| t^{d-1}\,  dt
 \leq \f{1} {|\mathbb{S}^{d-1}|} \|f_d^\va - f_d \|_{L^1}.
\end{align*}  Furthermore, by the triangle inequality, we have
\begin{align*}
\|f_d^\va - f_d \|_{L^1}& \leq  \|f_d\ast \phi_{\va}-f_d\|_{L^1} +\|(f_d-f_d \cdot \chi_{B_{1-\va}})\ast \phi_{\va}\|_{L^1} + \|(f_d \cdot \chi_{B_{1-\va}} )\ast \phi_{\va}-f_d^\va\|_{L^1}\\
&\leq     \|f_d\ast \phi_{\va}-f_d\|_{L^1}+\|f_d (1-\chi_{B_{1-\va}})\|_{L^1} +C \va^{\ld+1}.
\end{align*}
Now, letting $\va\to 0+$, we then have $\|f_d\ast \phi_{\va}-f_d\|_{L^1} \to 0$ by the norm convergence of approximation to the identity. The second and third terms go  to $0$ evidently. Hence, we deduce that  
\begin{align}\label{3.17}
&\Bl|\int_0^\t \(f^\va(\f t\t) -f(\f t\t)\) R_n^\ld (\cos t) (\sin t)^{2\ld}\, dt \Br| \leq \va.
\end{align}
Thus, letting $\va\to 0+$ in \eqref{3.17}, we obtain for all $n\in\mathbb{N}_0$,
\begin{equation}\label{5.10}
\int_0^\t f(\f t\t) R_n^\ld (\cos t) (\sin t)^{2\ld}\, dt \ge 0.
\end{equation}

Finally, we show that if $\wh{f_d}(\xi)>0$ for all $\xi\in \RR^d$, then \eqref{5.10} with strict inequality holds.  To see this, we need to use the \eqref{4.2} to obtain that for every $\ta\in (0,\pi),$
\begin{align*}
	I_n(\ta)&=\ta^{-2\ld-1}\int_0^\ta f(\f t\t) R_n^\ld (\cos t) (\sin t)^{2\ld} \, dt\\
	&= \ta^{-2\ld-1}2^{2\lambda}\int_0^\ta f(\f t\t) R_n^\ld (\cos t)(\cos \f t2)^{2\ld} (\sin \f t2)^{2\ld} \, dt\\
	& = \ta^{-2\ld-1}2^{2\lambda}\int_0^\ta f(\f t\t) \sum_{j=n}^{n+\ld} a_{n,j} ^\ld R_{2j}^{\ld}(\cos \f t2)   (\sin \f t2)^{2\ld} \, dt\\
	&=\ta^{-2\ld-1}2^{2\lambda+1} \sum_{j=n}^{n+\ld} a_{n,j} ^\ld 	I_{2j}(\f \ta 2).
\end{align*} Thus, for a fixed $n$, if we repeat the same process $m$-times, we have

\begin{align}\label{3.20}
I_n(\ta)&=(\ta^{-2\ld-1}2^{2\lambda+1})^m \sum_{j_1=n}^{n+\ld}\sum_{j_2=2j_1}^{2j_1+\ld}\sum_{j_3=2j_2}^{2j_2+\ld}\cdots \sum_{j_m=2j_{m-1}}^{2j_{m-1}+\ld} a_{n,j_1}^\ld  a_{2j_1,j_2}^\ld  \cdots a_{2j_{m-1},j_m}^\ld	I_{2j_m}(\f \ta {2^m}).
\end{align} It follows from \eqref{5.10} that each term in \eqref{3.20} is nonnegative. We can conclude that for any $m\in \NN$,
$$I_n(\t)\ge  c_{n,m} I_{2^m n} (2^{-m}\t ),$$
where $c_{n,m}>0$ for any $m\in\NN$. Noticing that $2^mn\cdot 2^{-m}\t=n\t$, which satisfies the case in Section \ref{case3}, letting $m\to\infty$ and following the same argument in Case (iii), we can conclude that
\begin{align*}
	I_{2^mn}(2^{-m}\t)&=N^{-1}2{m(2\l+1)}\int_0^{n+\lambda2^{-m}}f(\f {t}{N+2^{-m}\l}) R_{2^mn}^\l(\cos \f tN)(\sin \f tN)^{2\l} dt + O\((2^mn)^{-1}\)\\
	&\geq c_\l \int_{0}^1 f(x) j_{\l-\f 12}\((n+2^{-m}\l)x\) x^{2\l} dx - C_n 2^{-m}\\
	&\geq c_\l \min_{n\leq u\leq n+\l}\int_{0}^1 f(x) j_{\l-\f 12}(ux) x^{2\l} dx - C_n 2^{-m},
\end{align*} 
where $x=\f {t}{N+2^{-m}\l}$ and $u=n+2^{-m}\l$. The last step is positive for a sufficiently large integer $m$. This in turn implies  the strict inequality in \eqref{5.10}.
  
\section{Proof of Corollary \ref{cor1.5} and Theorem \ref{thm1.4}}
In this section, we will show the Corollary \ref{cor1.5} and Theorem \ref{thm1.4} respectively.

\begin{proof}[Proof of Corollary \ref{cor1.5}.]
Since $g$ is an  isotropic   continuous  positive semi-definite function on $\SS^d$ with   compact support on $[0,\pi]$, by the Schoenberg's theorem and Lemma \ref{lem2.3}, it can be easily seen that for each positive integer $\ell$, and every nonnegative integer $n$,
the function
$$ R_n^\ld (\cos 2^{\ell} t) \Bl( \f { \sin (2^\ell t)}{2^\ell \sin t} \Br)^{2\ld}$$
can be expressed as a  combination of the polynomials $R_j^\ld(\cos t)$, $j\ge 0$. It follows that for any $\ell\in\NN$,
\begin{equation*} 
g(2^\ell t)  \Bl( \f { \sin (2^\ell t)}{ \sin t} \Br)^{2\ld}
\end{equation*}
is an  isotropic   continuous positive semi-definite function on $\SS^d$.

Let $x>0$ and $\ell\in \NN$. Set $n=n_{x,\ell}\in\NN$ be such that $\f {n-1} {2^\ell} \leq x < \f {n}{2^{\ell}}.$ Then by the above claim,
\begin{align*}
0&\leq n^{2\ld+1}2^{-2\ell \ld}  \int_0^{2^{-\ell } \pi} g(2^\ell t)  \(\sin (2^\ell t)\)^{2\ld} R_n^\ld (\cos t) \, dt \\
&= \int_0^{2^{-\ell } n\pi} g(2^\ell t n^{-1}) \Bl (\f{ \sin (2^\ell  t n^{-1})}  {2^\ell n^{-1} t}\Br)^{2\ld} R_n^\ld \(\cos (\f tn)\)  t^{2\ld}\, dt \\
&=\int_0^{x\pi} g(x^{-1} t) \Bl (\f{ \sin (x^{-1} t )}  {x^{-1}  t}\Br)^{2\ld} R_n^\ld \(\cos (\f tn)\)  t^{2\ld}\, dt+o (1),\
\end{align*}
where the last step uses the fact that $|x^{-1}-2^\ell n^{-1}|\leq \f 1{nx}\to 0$ as $n\to \infty$.  Thus, letting $n\to\infty$ and applying the Proposition \ref{lem3.1}, we get
\begin{align*}
0&\leq x^{-2\ld-1}\int_0^{\pi x} g(x^{-1} t) \Bl (\f{ \sin (x^{-1} t )}  {x^{-1}  t}\Br)^{2\ld} j_{\ld-\f12}  (t)  t^{2\ld}\, dt\\
&=\int_0^\infty  g(t) \Bl (\f{ \sin t}  { t}\Br)^{2\ld} j_{\ld-\f12}  (xt)  t^{2\ld}\, dt\\
&=c_\ld \wh{G_d} (\xi),
\end{align*}
where, by the notation in Proposition \ref{prop2.5}, $\xi\in\RR^d$, $|\xi|=x$ and
$$ G(t):=g(t) \Bl( \f {\sin t} t\Br)^{2\ld},\   \ t\ge 0.$$
This shows that $G$ is an isotropic continuous  positive semi-definite function on $\RR^d$ by    Bochner's theorem. Furthermore, when $d$ is odd, and by Theorem \ref{thm1.3},  $G$ is also  a  continuous positive semi-definite function on $\sph$. 
\end{proof}

\begin{proof}[Proof of Theorem \ref{thm1.4}.]
	Fix $x>0$.  Let $n\ge 2 x$ be an integer and let $\ta =\ta_n=\f x {n}$.
	Since $g(\theta^{-1}\cdot)$ is an isotropic positive semi-definite function on $\SS^d$  and $g$ is supported on $[0,\pi]$, 
	\begin{align*}
	0\leq   &n^{2\ld+1}  \int_0^{\ta \pi} g (\ta^{-1} t) R_n^\ld (\cos t) (\sin t)^{2\ld} \, dt  = \int_0^{x\pi} g\Bl( \f { t} x\Br) R_n^\ld (\cos \f tn) \Bl( \f {\sin (n^{-1} t)}{n^{-1} t}\Br)^{2\ld} t^{2\ld}\, dt.
	\end{align*}
	By  Proposition~\ref{lem3.1} (see also \cite[Theorem 8.21.12]{Sz}), letting $n\to\infty$ and $u=\f tx$ , we obtain
	\begin{align*}
 \lim_{n\to\infty} \int_0^{x\pi}  g\Bl( \f { t} x\Br) R_n^\ld (\cos \f tn) \Bl( \f {\sin (n^{-1} t)}{n^{-1} t}\Br)^{2\ld} t^{2\ld}\, dt    & =  \int_0^{x\pi} g\Bl(\f {t} x\Br)j_{\ld-\f12} (t) t^{2\ld}\, dt \\ &=x^{2\ld+1} \int_0^\infty  g(u) j_{\ld-\f12} (u x) u^{2\ld}\, du\\& =c_\l x^{2\ld+1} \wh{g_d} (\xi)\geq0,
	\end{align*} 
where, by the notation in  Proposition \ref{prop2.5}, $\xi\in\RR^d$, and $|\xi|=x$.
	By  Bochner's theorem and the arbitrariness    of $x$, we prove that $g$ is  a  positive semi-definite function on $\RR^d$.
\end{proof}

\section{The case of dimension $d=2$}\label{section5}
In this section, we deal with the Conjecture  \ref{conj1.9}.  Since the convergence of series   $\sum_{n=0}^{\infty}a_nC_n^{\l}(1)<\infty$ can be verified by using the same argument in Theorem \ref{thm1.3}, we only need to   consider the following conjecture: 

\begin{conj} \label{conj5.1}
	Let  $\delta\geq \l+1$ and $\lambda=\f{d-1}2 $.  Then for any $\t \in (0,\pi)$ and $n \in \mathbb{N}_0,$

 	\begin{equation*} 		 
 		\int_0^\t (\t-t)^{\delta}C^\l_n(\cos  t)\sin^{2\l} t d t> 0.		
 	\end{equation*}
\end{conj}  Actually, we can claim that to prove  Conjecture \ref{conj5.1},  it suffices to consider the boundary case $\delta=\l+1$.  
 In fact, notice that for $a>-1, b>0$,  by using the substitution $t=u+s(\t-u)$, we have \begin{align*} \int_{0}^{\t}(\t-t)^a\int_{0}^t(t-u)^b g(u) du dt & =  \int_{0}^{\t}g(u)\int_{u}^\t(t-u)^b(\t-t)^a dt du\\&  =  \int_{0}^{\t}g(u)\int_{0}^1(\t-u)^{a+b+1}(1-s)^a s^b ds du\\ &=B(b+1,a+1)  \int_{0}^{\t}g(u) (\t-u)^{a+b+1}du, \end{align*} where $B(x,y)=\int_{0}^1t^{x-1}(1-t)^{y-1}dt$ is the Beta function. Then  we have \begin{align*} 	\int_{0}^{\t}g(u) (\t-u)^{a+b+1}du=\frac{1}{B(b+1, a+1)}\int_{0}^{\t}(\t-t)^a\int_{0}^t(t-u)^b g(u) du dt. \end{align*} Now for $\delta_1> \delta_2 $,  let $a=\d_1-\d_2-1$, $b=\delta_2$, and $g(u) =C_n^\l(\cos u)(\sin u)^{2\l}$. It follows immediately that  \begin{align*}  \int_{0}^{\t}  (\t-u)^{\delta_1}& C_n^\l(\cos u)(\sin u)^{2\l} du\\&=\frac{1}{B(\delta_2+1,\delta_1-\delta_2)}\int_{0}^{\t}(\t-t)^{\delta_1-\delta_2-1}\int_{0}^t(t-u)^{\delta_2} C_n^\l(\cos u)(\sin u)^{2\l}  du dt. \end{align*} Hence, we will focus on the boundary case.
 
Our main result is to show that  Conjecture \ref{conj5.1} is true when the dimension $d=2$ under the restriction when $\t$ is   small.  By the argument above,  it suffices to consider the boundary case $\delta=\f32$. More precisely, we state  the main  theorem  below.

\begin{thm}\label{thm5.3}
The function
	$$f_{\t,\delta}(t)=(\t-t)_+^\f32$$    is isotropic positive definite on $\mathbb{S}^2$ whenever $0\leq\t\leq c$,
	 where $c$ is an absolute constant $0 < c < \pi$.
\end{thm}
\begin{rem}
	In the proof, we  give  an upper estimate of   $c$ and it can be taken at most  $1.2644 \times 10^{-21}$. We believe this value could be improved.
\end{rem}
For the proof of Theorem \ref{thm5.3},  by  the above analysis according to the  Schoenberg's theorem, it suffices to show that 
\begin{align}\label{5.1}
 & \int_0^\t (\t-t)^{\f 32}P_n(\cos t)\sin t dt > 0,\quad \forall \, n \in \mathbb{N}_0.
 \end{align}
The proof of the positivity in \eqref{5.1} consists of three cases:  (i) $n\t\geq A $;  (ii) $n\t\leq B $;   (iii) $B\leq n\t\leq A$.  
Next, we will give  detailed proofs of the three cases. 
\subsection{Case (i).} In this case, we shall prove that there exists a constant $A>1$  such that \eqref{5.1} holds whenever $n\ta \ge A$.  The proof is long and will
be divided into several subsections.

\subsubsection{Decomposition of the integral.}

We first give the following decomposition on the integral \eqref{5.1}:
\begin{lem} 
	For $\t\in (0, \f\pi 2]$ and $n\ge 1$, 
	\begin{align*}
	\f {4n(n+1)}3 &\int_0^\t (\t-t)^{\f 32} P_n(\cos t) \sin  t\, d t=I_{n,1}(\t) +R_{n,1}(\t)+R_{n,2}(\t)-R_{n,3}(\t),\end{align*}
	where
	\begin{align*}
	I_{n,1}(\t)&= \f2{n(n+1) } \int_0^\t \Bl[ 1-P_{n}(\cos t)\Br]\f { \sqrt{\t-t}\cos t}{\sin^2 t}\, d t,\\	
	R_{n,1}(\t)&=  \f {1}{2n} \int_0^\t P_{n-1}^{(1,1)} (\cos t)\f{ \sin t \cos t}{\sqrt{\t-t}} \, d t,\\
	R_{n,2}(\t)&=\f 1{n(n+1) } \int_0^\t \Bl[ 1-P_{n}(\cos t)\Br]  \f 1 {(\sin t) \sqrt{\t-t}} \, dt\ge 0,\\
	R_{n,3}(\t)&=\int_{0}^\t P_n(\cos t)\f {\sin t}{ \sqrt{ \t-t}}\, dt.\\
	\end{align*}

\end{lem}

\begin{proof}  Using \eqref{1-3}  and integration by parts, we obtain  that 
	\begin{align*}A_n:=&	\int_0^\t (\t-t)^{\f 32} P_n(\cos t) \sin t\, d t= -\f 2n \int_0^\t (\t-t)^{\f 32} \Bl( P_{n+1}^{(-1,-1)} (\cos t)\Br)' \, dt\\
	& =- \f 3n\int_0^\t P_{n+1}^{(-1,-1)} (\cos t) (\t-t)^{\f12}\, dt= \f 3{4n}\int_0^\t P_{n-1}^{(1,1)} (\cos t)\sin^2t (\t-t)^{\f12}\, dt,\end{align*}
	where we used \eqref{1-3} in the first step, and  \eqref{1-7} in the third step.
	Applying integration by parts to this last integral  once again yields
	\begin{align*}
	A_n&=-\f 32 \f 1{n(n+1)} \int_0^\t (\t-t)^{\f12} \Bl(P_{n} (\cos t)\Br)'\sin t\, dt \\
	&=-\f 34 \f 1{n(n+1)}\int_0^\t P_{n} (\cos t) \f{ \sin t}{\sqrt{\t-t}}\, dt +  \f 32 \f 1{n(n+1)}\int_0^\t P_{n} (\cos t)(\t-t)^{\f12} \cos t\, d t\\
	&=: A_{n,1} +A_{n,2}.  \end{align*}
	For the term $A_{n,2}$, we use integration by parts to obtain 
	\begin{align*}
	A_{n,2}
	&=-\f 3{n^2(n+1)}\int_0^\t \Bl(P_{n+1}^{(-1,-1)} (\cos t)\Br)' (\t-t)^{\f12} \cot t\, d t\\
	&=-\f 3{n^2(n+1)}\int_0^\t P_{n+1}^{(-1,-1)} (\cos t)\Bl[\f 12 (\t-t)^{-\f12} \cot t + \f{ (\t-t)^{\f12}}{\sin^2 t} \Br]\, d t\\
	&=\f 3 {8n^2(n+1)}\int_0^\t P_{n-1}^{(1,1)} (\cos t)\f{\sin t\cos t}{\sqrt{\t-t}}\, dt   + \f 3 {4n^2(n+1)}\int_0^\t P_{n-1}^{(1,1)} (\cos t) (\t-t)^{\f12} \, d t,
	\end{align*}
	where the third  step uses \eqref{1-7}. Putting the above together, we then obtain 
	\begin{align*}
	& \f {4n(n+1)}3  A_n=-\int_0^\t P_{n}(\cos t)\f{ \sin t}{\sqrt{\t-t}} d t\\
	& + \f {1}{2n} \int_0^\t P_{n-1}^{(1,1)} (\cos t) \f{\sin t \cos t}{\sqrt{\t-t}}\, dt   + \f {1}{n} \int_0^\t P_{n-1}^{(1,1)} (\cos t) (\t-t)^{\f12} \, dt\\
	&=-R_{n,3}(\t) +R_{n,1}(\t) + \f {1}{n} \int_0^\t P_{n-1}^{(1,1)} (\cos t) (\t-t)^{\f12} \, dt.
	\end{align*}
	For the last integral, we apply  integration by parts again to obtain  
	\begin{align*}
	\f 1 {n} &\int_0^\t P_{n-1}^{(1,1)} (\cos t) (\t-t)^{\f12}\, dt=\f 2 {n (n+1)} \int_0^\t \f {(\t-t)^{\f12} }{\sin t} \Bl(P_{n}(1)- P_{n}(\cos t) \Br)'\, d t\\
	&= \f2{n(n+1) } \int_0^\t \Bl[ 1-P_{n}(\cos t)\Br]\f { \sqrt{\t-t}\cos t}{\sin^2 t}\, dt+
	\f 1{n(n+1) } \int_0^\t \Bl[ 1-P_{n}(\cos t)\Br]  \f 1 {\sin t \sqrt{\t-t}} \, dt\\
	&=I_{n,1}(\t) +R_{n,2}(\t).\end{align*}
	This completes the proof of the lemma.
\end{proof}

\subsubsection{Estimates of      $I_{n,1}(\t)$, $R_{n,3}(\t)$, and $R_{n,1}(\t)$}
In this subsection, we will give the upper bounds for $R_{n,1}(\t)$ and $R_{n,3}(\t)$ (see Lemma \ref{lem2.4} and \ref{lem2.5}), and  the lower bound for $I_{n,1}(\t)$(see Lemma \ref{lem2.6}).

\begin{lem}\label{lem2.4} If $n \t\ge 5$ and $\t\in (0, \f\pi 2]$, then 
	\begin{align*}
	|	R_{n,1}(\t)|&=  \f {1}{2n}\Bl| \int_0^\t P_{n-1}^{(1,1)} (\cos t) \f{\sin t \cos t}{\sqrt{\t- t}} \, d t\Br|\leq C (n\t )^{-2} \t^{\f 32},
	\end{align*}
	where the constant $C$ can be taken to be $42.6909$.
\end{lem}

\begin{proof}
 
	We break the integral $\int_0^\t$ into two parts $\int_{\t-n^{-1}}^\t+\int_{0}^{\t-n^{-1}}$.
	For the first part, we use  estimate    \eqref{C.3}, $\arctan x\leq x$, and  the condition $n\t\geq 5$ to get
	\begin{align*}
	\Bl|\f {1}{2n}  \int_{\t-\f1n}^\t P_{n-1}^{(1,1)}  (\cos t) \f{\sin t \cos t}{\sqrt{\t- t}} \, d t\Bl|  &\leq  {\f {2.821\times \pi^{\f32}} {2n(n-1)^{\f12}}} \int_{\t-\f1n}^\t (\t-t)^{-\f12} t^{-\f12}\, dt\\
	&\leq {\f {2.821\times \pi^{\f32}} {2n(n-1)^{\f12}}}\cdot 2\arctan\(\f{\f1n}{\t-\f1n}\)^{\f12}\\
	& \leq{\f {2.821\times \pi^{\f32}} { n(n-1)^{\f12}}} \cdot \(\f{\f1n}{\t-\f1n}\)^{\f12}\\ &\leq 2.821\times \pi^{\f32}\times \(\f{5}{4-\pi}\)^{\f12}n^{-2}\t^{-\f12}\\ &\leq 37.9275 \times n^{-2}\t^{-\f12}.
	\end{align*}
 
 	For the second part,  we use    \eqref{1-3}, integration by parts,    triangle inequality,  \eqref{Legendre estimates},  and $n\t\geq 5$ to obtain 
	\begin{align*}
	\Bl| \f {1}{2n} &\int_{0}^{\t-\f1n}  P_{n-1}^{(1,1)} (\cos t) \f{\sin t \cos t}{\sqrt{\t- t}} \, d t \Bl|=  \Bl|  \f{-1}{ n(n+1)}  \int_0^{\t-\f1n} (\t-t)^{-\f12}(\cos t)   \, d \Bl( P_n(\cos t)\Br) \Br|\\ &
	\leq \f1{ n^2} \Bl|P_n(\cos(\t-\f1n))\f{\cos(\t-\f1{n})}{\sqrt{\f1n}}\Br| +  \f1{ n^2}\f{1}{\t^{\f12}} +   \f1{2n^2}\Bl|\int_0^{\t-\f1n}P_n(\cos  t)(\t-t)^{-\f32} (\cos t)dt \Br| \\ &+ \f1{  n^2}\Bl|\int_0^{\t-\f1n}P_n(\cos  t)(\t-t)^{-\f12} (\sin t)dt \Br|  \\ &\leq
  \f1{n^2}   \f{1}{n^\f12(\t-\f1n)^\f12} \f{1}{\sqrt{\f1n}} 
 +\f1{ n^2\t^{\f12}} + \f{1}{2n^2}\int_{0}^{\t-\f1n}\f{1}{n^\f12t^\f12}(\t-t)^{-\f32}dt+ \f{1}{ n^2} \int_{0}^{\t-\f1n} \f{1}{n^\f12t^\f12} (\t-t)^{-\f12} t  dt \\& \leq \f{\sqrt{5}}{2}\f{1}{n^2\t^{\f12}} + \f 1{ n^2\t^{\f12}} +   \f{1}{ n^2\t^{\f12}} +  \f{2(\f\pi2)^2}{3n^2\t^\f12} \leq \f{4.7634}{n^2\t^{\f12}}.
	\end{align*} 
 Therefore, combining the above two parts, we get the desired estimate.
  
\end{proof}

\begin{lem}\label{lem2.5}

If $n \t\ge 5$ and $\t\in (0, \f\pi 2]$, then 
  	\begin{align*}	  |	R_{n,3}(\t)| =\left|  \int_{0}^\t P_n(\cos t)\f {\sin t}{ \sqrt{\t-t}}\, d t 
		\  \right| \leq \f { \sqrt{2}}{n\t} \sin (N\t)  \Bl( \f {\sin \t}{\t}\Br)^{\f12}\t^\f32				
		+   C \times (n\t)^{-\f32} \t^{\f32},	 
	\end{align*} where $N=n+\f12$, and 	 the constant $C$ can be taken to be  $92.1237.$
\end{lem}
 
\begin{proof}
	 We will need  a function $\eta\in C^1(\RR)$  where \begin{eqnarray}\eta(x)=
		\begin{cases}
			1, &\f12\leq x  \cr \f12\cos(4\pi x)+\f12, &\f14 \leq x < \f12\cr 0, &x<\f14\end{cases}.\notag
	\end{eqnarray} 
   Then, \begin{equation*}
	\int_{0}^\t P_n(\cos t)\f {\sin t}{ \sqrt{\t-t}}\, d t =
	 \int_{0}^{\f \t2} (1-\eta(\t^{-1}t))P_n(\cos t)\f {\sin t}{ \sqrt{\t-t}}\, d t+\int_{\f \t4}^{\t} \eta(\t^{-1}t) P_n(\cos t)\f {\sin t}{ \sqrt{\t-t}}\, d t.
	\end{equation*} For the first integral, by \eqref{1-8}  and  integration by parts, we have \begin{align*}
	\int_0^{\f\t2} (1-\eta(\t^{-1}t)) P_n(\cos t) \f{\sin t}{\sqrt{\t-t}} \, dt 
	& =- \int^{\f\t2}_{0}  (1-\eta(\t^{-1}t))\f{1}{\sqrt{\t-t}}d\Bl(\f{\sin^2 t}{2n}P_{n-1}^{(1,1)}(\cos t)\Br)\\
	& =- \f1{4n} \int_0^{\f \t2} P_{n-1}^{(1,1)} (\cos t)(\sin  t)^2 (\t-t)^{-\f32}g(t)dt,
	\end{align*} where $g(t) := 1-\eta(\t^{-1}t) -2\t^{-1} (\t-t) \eta'(\t^{-1}t).$
	Using \eqref{C.3} and noting that $|g(t)|\leq 4\pi+2$, we get 
	\begin{align} 
	\Bl|&	\int_0^{\f\t2}  (1-\eta(\t^{-1}t)) P_n(\cos t) \f{\sin t}{\sqrt{\t-t}}  dt \Br| \notag \\& =\Bl| \f 1{4n} \int_0^{\f \t2} P_{n-1}^{(1,1)} (\cos t)(\sin  t)^2 (\t-t)^{-\f32}g(t) \, dt \Br|\notag \\
	&\leq  \f {2.821\cdot \pi^{\f32}}{4n}  \int_0^{\f \t2}    \f 1{ (n-1)^{1/2}t^{3/2}} ( \sin t  )^2  (\t-t)^{-\f 32} |g(t)|dt\notag\\
	&\leq  \f {2.821 \cdot \pi^{\f32} \cdot (4\pi+2)}{4n(n-1)^{1/2}}  \int_0^{\f \t2}    \(\f {\sin t}{t}\)^{\f32}  (\sin t)^{\f12}    (\t-t)^{-\f 32}  dt\notag\\
	&\leq  \f {2.821 \cdot \pi^{\f32} \cdot(4\pi+2)}{4n(n-1)^{1/2}}   \(\f\t2\)^{\f12}    \int_0^{\f \t2}      (\t-t)^{-\f 32}  dt\notag\\
			&\leq  \f {2.821\cdot\pi^{\f32}\cdot (4\pi+2)\cdot \sqrt{2}}{4n ^{3/2}},\label{eqn5.4}
	\end{align} where the last inequality follows by $n\geq 2$ when $\t$ is small. 

Next, we estimate the second integral. We claim that \begin{align}\label{2-1}
	\f 1{\t^{\f 32}}  \int_{\f\t4} ^\t&  P_n(\cos t)\f{\eta (\t^{-1}t ) \sin t }{\sqrt{\t-t}} \, dt 
	\leq \f { \sqrt{2}}{n\t} \sin (N\t)  \Bl( \f {\sin \t}{\t}\Br)^{\f12}				
	+ \Bl(\f{1}{(2\pi)^{\f12}(\f2\pi)^{\f32}}+10.4160 \Br)(n\t)^{-3/2},	
	\end{align}
	where $N=n+\f12$. 
	To see this, we use   \eqref{propA.4}  
	to obtain    
	\begin{align*}
	\int_{\f\t4} ^\t&  P_n(\cos t)\f{\eta (\t^{-1}t ) \sin t }{\sqrt{\t-t}} \, dt\\
	&\leq 2^{1/2} \pi^{-\f12} n^{-\f12} \int_{\t/4} ^\t \eta (\t^{-1}t)  \cos (Nt -\f \pi4)(\sin t)^{\f12} (\t-t)^{-\f12}\, dt
	+\f{1}{(2\pi)^{\f12}(\f2\pi)^{\f32}}      (n\t)^{-\f32}\t^\f32 .
	\end{align*}
	We then write 
	\begin{align*}
	2^{1/2}& \pi^{-\f12} n^{-\f12} \int_{\t/4} ^\t \eta (\t^{-1}t)  \cos (Nt -\f \pi4)(\sin t)^{\f12} (\t-t)^{-\f12}\, dt\\&= 2^{- 1/2} \pi^{-\f12} n^{-\f12} \int_{\t/4} ^\t \eta (\t^{-1}t) \Bl[ e^{i(Nt -\f \pi 4)} +e^{i(-Nt +\f \pi4)}\Br](\sin t)^{\f12} (\t-t)^{-\f12}\, dt\\
	&=:I_n^{+} + I_n^{-}.
	\end{align*}
	Using Proposition \eqref{lem-asym} and \eqref{lem-asym2} with $\phi(t)=\eta(\frac{t}{\t})(\sin t)^{\f12}$,  
	$\t\in(0,\f\pi 2),$  $\a=\f{\t}4$, $\b=\t$, $\l=\f12$, $\mu=\f12$, and  $v=1$,  we  have
	\begin{align*} 
		\int_{\f\t4} ^\t \eta (\t^{-1}t)    (\sin t)^{\f12} e^{iNt }(\t-t)^{-\f12}\, d t 	&\leq   \f{\Gamma(1/2)}{N^{\f 12}}e^{iN\t-\f 12\cdot \f 12\pi i}  \eta (1) (\sin \t)^{\f 12} + \f{13.0552}N,
	\end{align*}
	\begin{align*} 
		\int_{\f\t4} ^\t \eta (\t^{-1} t)    (\sin t)^{\f12} e^{-iNt} (\t-t)^{-\f12}\, dt 
		&\leq   \f{\Gamma(1/2)}{N^{\f 12}}e^{-iN\t+\f 12\cdot \f 12\pi i}  \eta (1) (\sin \t)^{\f 12} + \f{13.0552}N.
	\end{align*}
Therefore, we obtain that  
	\begin{align*}
	I_n^{+} &=2^{-\f12} \pi^{-\f12} n^{-\f12} e^{-\f {\pi i}4} \int_{\t/4} ^\t \eta (\t^{-1}t)  e^{iNt } (\sin t)^{\f12} (\t-t)^{-\f12}\, d t\\
	&\leq 2^{-\f12} \pi^{-\f12} n^{-\f12} e^{-\f {\pi i}4} \Big[\f{\Gamma(1/2)}{N^{\f 12}}e^{iN\t-\f 12\cdot \f 12\pi i}  \eta (1) (\sin \t)^{\f 12} + 13.0552\times N^{-1}  \Big] \\
		&\leq 2^{-\f12}  n^{-1}e^{iN\t} e^{-\f \pi2 i} (\sin \t)^{\f12} +  e^{-\f \pi2 i} \cdot 5.2080 \times (n\t)^{-\f 32} \t^{\f 32},
	\end{align*}
	and  
	\begin{align*}
	I_n^{-} &=2^{-\f12} \pi^{-\f12} n^{-\f12} e^{\f {\pi i}4} \int_{\t/4} ^\t \eta (\t^{-1} t)  e^{-iNt} (\sin t)^{\f12} (\t-t)^{-\f12}\, dt\\
	&\leq 2^{-\f12}  n^{-1} e^{\f {\pi i}2}  e^{-i N \t } (\sin \t)^{\f12}				
	+  e^{-\f \pi2 i} \cdot  5.2080 \times (n\t)^{-\f 32} \t^{\f 32}.
	\end{align*}
	It follows that 
	\begin{align*}
	|I_n^+ +I_n^-| \leq 2^{\f12}  n^{-1} \sin (N\t)  (\sin \t)^{\f12}				
	+ 10.4160\times  (n\t)^{-\f 32} \t^{\f 32}.
	\end{align*}
	This implies \eqref{2-1}.
 
 By equations \eqref{eqn5.4} and \eqref{2-1}, we thus conclude the desired estimate for $R_{n,3}(\t)$. 
\end{proof}

\begin{lem}\label{lem2.6}
	For $n\t\ge 5$ and $\t\in (0,\f \pi2]$,  we have that  
	\begin{align*}
	 	I_{n,1} (\t) = \f2{n(n+1) } \int_0^\t \Bl[ 1-P_{n}(\cos t)\Br]\f { \sqrt{\t-t}\cos t}{\sin^2 t} dt \ge \f{3}{2n}\sqrt{\t}- 30.1067\times (n\t)^{-\f32}\t^{\f32}.
	\end{align*}  
 \end{lem}

\begin{proof}
	Firstly, we claim the following lower estimate. 
	\begin{equation}\label{2-3}
	I_{n,1}(\t) \geq  \f34  \f{2\sqrt{\t} }{n^2}  \int_0^\t\f {1-P_n(\cos t)}{t^2}\, dt-\f3{2n^2}  \t^{-\f12} \log(n\t) \left( \f4\pi(4+\sqrt{\f\pi2})+4 \right).  	  	 	
	\end{equation} Indeed, since that $$\frac{\cos t}{\sin^2 t}=\frac{1-2\sin^2\f t2}{\sin^2 t}=\f{-2\sin^2\f t2}{\sin^2 t}+  \f1{t^2}\(  \f{t^2}{\sin^2 t}-1\)+\f{1}{t^2},$$
  we may   rewrite the integral $I_{n,1}(\t)$ to 	
	\begin{align*} 
\f {n(n+1)}2I_{n,1}(\t) 	=&- \int_0^\t \Bl[ 1-P_{n}(\cos t)\Br]\f {2 \sqrt{\t-t}\sin^2\f t2}{\sin^2 t}\, d t\\
	&+ \int_0^\t \f{ 1-P_{n}(\cos t)}{t^2}\sqrt{\t-t} \Bl[\f {t^2 }{\sin^2 t}-1\Br]\, dt+
	\int_0^\t\sqrt{\t-t}  \f{ 1-P_{n}(\cos t)}{t^2}\, dt. 
	\end{align*}	
Then we will prove that
		\begin{align} 
		\f {n(n+1)}2I_{n,1}(\t)  & \geq \sqrt{\t}\int_0^\t \f {1-P_n(\cos t)}{t^2}\, dt -\int_0^\t \f {1}{\sqrt{\t}+\sqrt{\t-t}}  \f {1-P_n(\cos t)}{t}\, dt -  2(\f2\pi)^2\t^\f32\label{eqn5.10}  \\
	\geq& \sqrt{\t}\int_0^\t\f {1-P_n(\cos t)}{t^2}\, dt- \left( \f4\pi(4+\sqrt{\f\pi2})+2 \right) \t^{-\f12} \log(n\t) - 2(\f2\pi)^2\t^\f32 \label{eqn5.11} \\
		\geq& \sqrt{\t}\int_0^\t\f {1-P_n(\cos t)}{t^2}\, dt-  \t^{-\f12} \log(n\t)  \left( \f4\pi(4+\sqrt{\f\pi2})+4 \right).\label{eqn5.12} 
	\end{align}
	The first  inequality \eqref{eqn5.10} follows that 
	\begin{align*}
	\Big| &- \int_0^\t \Bl[ 1-P_{n}(\cos t)\Br]\f {2 \sqrt{\t-t}\sin^2\f t2}{\sin^2 t}\, d t	+ \int_0^\t \f{ 1-P_{n}(\cos t)}{t^2}\sqrt{\t-t} \Bl(\f {t^2 }{\sin^2 t}-1\Br)\, dt+\\	&\int_0^\t\sqrt{\t-t}  \f{ 1-P_{n}(\cos t)}{t^2}\, dt -\sqrt{\t}\int_0^\t \f {1-P_n(\cos t)}{t^2}\, dt +\int_0^\t \f {1}{\sqrt{\t}+\sqrt{\t-t}}  \f {1-P_n(\cos t)}{t}\, dt \Big|\\
=&\Big|- \int_0^\t \Bl[ 1-P_{n}(\cos t)\Br]\f {2 \sqrt{\t-t}\sin^2\f t2}{\sin^2 t}\, d t	+ \int_0^\t \f{ 1-P_{n}(\cos t)}{t^2}\sqrt{\t-t} \Bl(\f {t^2 }{\sin^2 t}-1\Br)\, dt\Big|\\
 =& \Big|\int_0^\t \Bl[ 1-P_{n}(\cos t)\Br]\sqrt{\t-t} \Big (\f {-2  \sin^2\f t2}{\sin^2 t}+\f{1}{\sin^2 t}-\f{1}{t^2} \Big )\, d t\Big| \\
\leq& 2 \t^\f12 \int_0^\t  \Big|\f {-2  \sin^2\f t2}{\sin^2 t}+\f{1}{\sin^2 t}-\f{1}{t^2} \Big|\, d t
	\leq 2 \t^\f32 (\f2\pi  )^2.
	\end{align*}
	To show the second inequality \eqref{eqn5.11}, we need to split the integral $\int_0^\t \f {1}{\sqrt{\t}+\sqrt{\t-t}}  \f {1-P_n(\cos t)}{t}\, dt$ into two parts: $\int_{0}^{1/n}+\int_{1/n}^{\t}$. For the first part,  we  use   representation \eqref{1-5}. 
Without loss of generality, we may assume that $n=2m$ (The case when $n=2m-1$ can be treated similarly).  Then we can obtain   that 
\begin{align}\label{eqn5-7}  P_{2m}(\cos t) = \f 1{\pi} \Bl(\f{ \Ga(m+\f12)}{\Ga(m+1)}\Br)^2 +\f 2{\pi} \sum_{j=1}^m \f {\Ga(m-j+\f12)\Ga(m+j+\f12)}{\Ga(m-j+1) \Ga(m+j+1)}\cos (2j t).
	\end{align}	
By the identity 
\begin{align}\label{eqn5-8} 
1-\f 1\pi\(\f{\Gamma(m+\f 12)}{\Gamma(m+1)}\)^2=\f 2{\pi} \sum_{j=1}^m \f {\Ga(m-j+\f12)\Ga(m+j+\f12)}{\Ga(m-j+1) \Ga(m+j+1)},
\end{align} and  using Gautschi's inequality:  for $s\in(0,1)$, $x^{1-s}<\f{\Gamma(x+1)}{\Gamma(x+s)}< (x+1)^{1-s},$ 
we have
	\begin{align*}
		  \Big|\int_0^{1/n} &\f  {1}{\sqrt{\t}+\sqrt{\t-t}}   \f {1-P_n(\cos t)}{t}\, dt \Big|  \\& =  \f 4{\pi} \sum_{j=1}^m \f {\Ga(m-j+\f12)\Ga(m+j+\f12)}{\Ga(m-j+1) \Ga(m+j+1)}\int_{0}^{1/n} \f {1}{\sqrt{\t}+\sqrt{\t-t}}\f{\sin^2(jt)}{t}dt\\
		 &\leq\t^{-\f12}\f4{\pi}\Big[\sum_{j=1}^{m-1}\f{1}{\sqrt{m-j}\sqrt{m+j}}\int_{0}^{1/n}\f{\sin^2(jt)}{t}dt + \sqrt{\f{\pi}{2m}} \int_{0}^{1/n}\f{\sin^2(mt)}{t}dt \Big]\\
		 &\leq\t^{-\f12}\f4{\pi}\Big[\f1n  \sum_{j=1}^{m-1}\f{j}{\sqrt{m^2-j^2} }  + \sqrt{\f{\pi}{2}}  \Big]
		 \leq\t^{-\f12}\f4{\pi}\Big[\f1n  \int_{0}^m\f{x}{\sqrt{m^2-x^2} }dx  +\sqrt{\f{\pi}{2}} \Big]\\
		&\leq\t^{-\f12}\f4{\pi}\Big( 4 + \sqrt{\f{\pi}{2}} \Big)\leq
		\t^{-\f12}\log(n\t)\cdot\f4{\pi}\Big(4 + \sqrt{\f{ {\pi}}{2 }} \Big).
 	\end{align*}

While for the second part, we have 
	\begin{align*}
	\Big|\int_{1/n}^\t \f {1}{\sqrt{\t}+\sqrt{\t-t}} & \f {1-P_n(\cos t)}{t}\, dt \Big|\leq \f2{\sqrt{\t}}\int_{1/n}^\t\f1tdt=\f2{\sqrt{\t}}\log(n\t).
\end{align*}
Combining the two parts, we obtain the desired estimate  \eqref{eqn5.11}.

 To show   \eqref{2-3}, by \eqref{eqn5.12},  
 $n\t \geq 5$, and $\t\in(0,\f\pi2)$,  
 we then have \begin{align*} 
	 I_{n,1}(\t) 	\geq & \f{2 }{n(n+1)} \Bl[\sqrt{\t}\int_0^\t\f {1-P_n(\cos t)}{t^2}\, dt- \t^{-\f12} \log(n\t) \left( \f4\pi(4+\sqrt{\f\pi2})+4 \right)   \Big]\\
	 	\geq &\f34  \f{2\sqrt{\t} }{n^2}  \int_0^\t\f {1-P_n(\cos t)}{t^2}\, dt-\f3{2n^2}  \t^{-\f12} \log(n\t) \left( \f4\pi(4+\sqrt{\f\pi2})+4 \right) .  	 	
\end{align*} This claims \eqref{2-3}.

For the remaining, we will deal with the integral $\int_0^\t\f {1-P_n(\cos t)}{t^2}dt.$ Without loss of generality, letting $n=2m$ and applying the equations \eqref{eqn5-7} and \eqref{eqn5-8}, we  have \begin{align}\label{eqn5.8}
	& \f {2\sqrt{\t}} {n^2} \int_0^\t \f {1-P_n(\cos t)}{t^2}\, dt=\f 8{\pi}\f {\sqrt{\t}}{n^2} \sum_{j=1}^m \f {\Ga(m-j+\f12)\Ga(m+j+\f12)}{\Ga(m-j+1) \Ga(m+j+1)}\int_0^\t \f { \sin^2 (jt)}{t^2}\, dt.
	\end{align}
	Note that  for $j\ge 1$, 
	\begin{align}\label{eqn5.9}
	\int_0^\t& \f {\sin^2 (jt)}{t^2}\, dt =j\int_0^{j\t} \f {\sin^2 t}{t^2}\, dt\ge j \int_0^\infty \f {\sin^2 t}{t^2}\, dt -\f 1\t
	=\f \pi 2 j -\f 1\t.
	\end{align} Thus, by using \eqref{eqn5.8}, \eqref{eqn5.9}, and Gautschi's inequality:  for $s\in(0,1)$, $x^{1-s}<\f{\Gamma(x+1)}{\Gamma(x+s)}< (x+1)^{1-s},$  we have for $n\t\geq 5$,
	\begin{align}
	\f {2\sqrt{\t}} {n^2} \int_0^\t \f {1-P_n(\cos t)}{t^2}\, dt&
	\geq \f {4\sqrt{\t}}{n^2} \sum_{j=1}^m \f {j\Ga(m-j+\f12)\Ga(m+j+\f12)}{\Ga(m-j+1) \Ga(m+j+1)} - 4  n^{-2}\t^{-\f12}\notag\\
	&\geq \f {4\sqrt{\t}}{n^2} \sum_{j=1}^{m-1} \f {j}{\sqrt{m^2-j^2}}-  { 2\sqrt{\pi}\sqrt{\f\pi2}}n^{-\f32} - 4  n^{-2}\t^{-\f12}\notag\\
	&\ge \f {4\sqrt{\t}}{n^2}\int_0 ^{m-1}\f y { (m^2 -y^2)^{\f12} }\, dy   - { 2\sqrt{\pi}\sqrt{\f\pi2}}n^{-\f32}  -    \f{4}{\sqrt{5}} n^{-\f32}\notag\\
	&\ge \f {4\sqrt{\t}}{n^2}(m-\sqrt{2m-1}) -\Bl(2\sqrt{\pi}\sqrt{\f\pi2}+  \f{4}{\sqrt{5}}\Br) (n\t)^{-\f32} \t^{3/2} \notag\\ 
				&\geq\f {2 \sqrt{\t}}{n}   -\Bl(10\sqrt{\f\pi2}+ 2\sqrt{\pi}\sqrt{\f\pi2}   +  \f{4}{\sqrt{5}}\Br) (n\t)^{-\f32} \t^{3/2}. \label{eqn5.17}
	\end{align} Hence, by \eqref{2-3} and \eqref{eqn5.17}, we have 	\begin{align*} 
		I_{n,1}(\t) 
			 &	\geq   \f34   \f {2 \sqrt{\t}}{n}   -\f34   \Bl(10\sqrt{\f\pi2}+ 2\sqrt{\pi}\sqrt{\f\pi2}   +  \f{4}{\sqrt{5}}\Br) (n\t)^{-\f32} \t^{3/2}      -\f3{2n^2}  \t^{-\f12} \log(n\t) \left( \f4\pi(4+\sqrt{\f\pi2})+4 \right)   \\
				&\ge \f{3}{2n}\sqrt{\t}- 30.1067\times (n\t)^{-\f32}\t^{\f32}. 
  			 \end{align*}
  		 \end{proof}

\subsubsection{Estimate of the integral} We prove the following estimate:
\begin{lem} 
	For $n\t\ge 5$ and $\t\in (0,\f\pi2]$,
	\begin{align*}
	\f {4n(n+1)}{3\t^{\f 32} }	\int_0^\t (\t-t)^{\f 32} P_n(\cos t) \sin t \, dt \ge (\f32- \sqrt{2})(n\t)^{-1}- 164.9212\times(n\t)^{-\f32}.
	\end{align*}
	In particular, this implies that there  exists a determined  constant $A>1$ such that 
	$$ 	\int_0^\t (\t-t)^{\f 32} P_n(\cos t) \sin t \, d t>0$$
	whenever $n\t\ge A$. 
\end{lem}

\begin{proof}
	By the Lemma \ref{lem2.4}, \ref{lem2.5}, and \ref{lem2.6}, for $n\t\geq5$, we have 
	\begin{align*}
	\f {4n(n+1)}{3\t^{\f 32}}&	\int_0^\t (\t-t)^{\f 32} P_n(\cos t) \sin t \, dt\\
	 \ge& \f32(n\t)^{-1} -  30.1067 \times (n\t)^{-\f32}  -\f { \sqrt{2}}{n\t} \sin (N\t)  \Bl( \f {\sin \t}{\t}\Br)^{\f12}   \\&-  92.1237\times(n\t)^{-\f 32}  - 42.6909  \times (n\t)^{-2}  \\
	 \geq& (\f32- \sqrt{2})(n\t)^{-1}-164.9213\times(n\t)^{-\f32}>0 ,
	\end{align*}  
whenever \begin{align*} 
 n\t\geq  \({\frac{164.9213}{\f32-\sqrt{2}}}\)^2> 3.6959\times 10^6.
\end{align*} This means there exists a determined positive constant $A$ such that the desired integral \eqref{5.1} is positive when $n\t\geq A$. 
\end{proof}

\subsection{Case (ii).}
In this case, we shall prove that there exists a constant $B\in (0,1)$  such that \eqref{5.1} is true whenever $n\ta\leq B$. 

To see this, we first note that
\begin{align}\label{eqn5.19}
\int_0^\t (\t-t)^{\f 32} P_n(\cos t) \sin t \, d t=\t^{\f 32}\int_0^\t (1-\f t \t)^{\f 32} P_n(\cos t) \sin t \, d t. 
\end{align} 
Following the same idea in the Case (ii) of odd dimensions,  we may rewrite \eqref{eqn5.19} as  
\begin{align*}
	 \int_0^\t (1-\f t \t)^{\f 32} P_n(\cos t) \sin t \, d t 
	&  = \int_0^\t  (1-\f t \t)^{\f 32} t\, dt + \int_0^{\ta}  (1-\f t \t)^{\f 32}  \Bl( P_n  (\cos t) \f {\sin  t} {t } -1\Br)t \, dt.
\end{align*}
Note that for any $t\in [0,\t]$,
\begin{align*}
	|P_n(\cos t)\f{\sin t}{t}-1|&\leq|P_n(\cos t)(\f{\sin t}{t})-\f{\sin t}{t}|+|\f{\sin t}{t}-1|\\
	&\leq|\f{\sin t}{t}|\cdot|P_n(\cos t)-1|+t\\&\leq|P_n(\cos t)-P_n(\cos 0)|+nt\\
	&\leq nt||P_n(\cos t')||_{\infty}+nt<2nt,
\end{align*}
and the second last inequality follows from Bernstein's inequality for trigonometric polynomials.
Then we have

 \begin{align*}
\ta^{-2 }\int_0^\t (1-\f t \t)^{\f 32} P_n(\cos t) \sin t \, d t
&\ge \ta^{-2 }\int_0^\t  (1-\f t \t)^{\f 32}t \, dt - 2 n \ta^{-2 } \int_0^\ta t^{2}\, dt
> \f 4{35} - 2n\ta  , 
\end{align*}
which implies that the desired integral \eqref{5.1} is positive whenever $n\ta< \f2{35}$.

\subsection{Case (iii).}\label{case 3}

In this case, we shall prove \eqref{5.1} for the remaining case $B \leq n\ta \leq A$.

We write
$$\int_0^\t (\t-t)^{\f 32} P_n(\cos t) \sin t \, d t=\t^{\f 32}\int_0^\t (1-\f t \t)^{\f 32} P_n(\cos t) \sin t \, d t.$$ 
To simplify our calculation, we will next show  $$\t^{-2}\int_0^\t (1-\f t \t)^{\f 32} P_n(\cos t) \sin t \, d t>0.$$
Let $N=n+\f12$, and by substitution $t=\f{t'}{N}$, we have $$ \ta^{-2 }\int_0^\t (1-\f t\t)^{\f 32} P_n(\cos t) \sin t \, d t = N^{-1} \ta^{-2 } \int_0^{N\ta} (1-\f {t'}{N\t})^{\f 32} P_n  (\cos\f {t'} N) \Bl(\sin \f {t'} N\Br) \, d{t'}.$$
By using  \eqref{C.1}, and taking   substitution $x=\frac{t'}{N\t}$, we have 
 
\begin{align}
 N^{-1} \ta^{-2 }  \int_0^{N\ta} &(1-\f {t'}{N\t})^{\f 32}  P_n  (\cos\f {t'} N) \Bl(\sin \f {t'} N\Br) \, d{t'}\notag\\
&\geq N^{-1} \ta^{-2 }  \int_0^{N\ta} (1-\f {t'}{N\t})^{\f 32}  \Bl(\frac{t'}{N}\Br)^{}\Bl(\f{\sin\frac{t'}{N}}{\frac{t'}{N}}\Br)^{\f12}  j_0(t') \, d{t'} - 0.1711\times n^{-1}\notag\\
&=  \int_0^1 (1-x)^{\f 32} j_{0} (N\ta x) x  \Bl( \f {\sin  \ta x }{\ta x} \Br)^{\f12} \, dx -  0.1711\times  n^{-1} \notag\\
&\ge  \sqrt{\f2\pi}   \int_0^1 (1-x)^{\f 32} j_{0} (N\ta x) x \, dx - 0.1711\times n^{-1}\notag \\
&\ge    
 \sqrt{\f2\pi} \int_0^1 (1-x)^{\f 32} j_{0} (ux) x^{ }\, dx  - 0.1711\times  B^{-1}\ta , \label{5.8}
\end{align} where $u:=N\t$. To  make the last expression   positive, we will find a positive lower bound for the integral $\int_0^1 (1-x)^{\f 32} j_{0} (ux) x  dx$ when $B \leq u\leq A$. By the reference \cite[Corollary 1.1, Page 551, 552, 556]{FI},  it shows

 \begin{eqnarray} \int_0^1 (1-x)^{\f 32} j_{0} (ux) x  dx\geq\f{\Gamma(2)\Gamma(\f52)}{\Gamma(1)\Gamma(\f92)} 
 	\begin{cases}
 		1-\f{2}{33}u^2, &0\leq u \leq \sqrt{\f{33}2}\approx 4.0620 \cr 
 		\f{2.0963}{u^3}, & u \geq \sqrt{12}\approx 3.4641 
 	\end{cases}.
 \end{eqnarray}
 
Thus, in our case, for $B \leq u\leq A$, we have to separate two cases to find the positive lower bound:

(1) When $B \leq u \leq \sqrt{12}$, we have 
\begin{align*}
	 \int_0^1 (1-x)^{\f 32} j_{0} (ux) x  dx \geq \f{\Gamma(2)\Gamma(\f52)}{\Gamma(1)\Gamma(\f92)}  (1-\f{24}{33})=\f{\Gamma(2)\Gamma(\f52)}{\Gamma(1)\Gamma(\f92)}\f{3 }{11}.
\end{align*}

(2) When $\sqrt{12}\leq u \leq A$, we have 
\begin{align*}
 \int_0^1 (1-x)^{\f 32} j_{0} (ux) x  dx  \geq \f{\Gamma(2)\Gamma(\f52)}{ \Gamma(1)\Gamma(\f92)} \f{2.0963}{A^3}.  
\end{align*} Notice that $\f{3 }{11}>\f{2.0963}{A^3}$. We will use (2) as the   positive lower bound.  Hence, the last step \eqref{5.8} is positive  if we assume $\t$ satisfying the following condition:

\begin{align*}
	\t<&\f{B}{0.1711}\cdot  \sqrt{\f2\pi}   \cdot \f{\Gamma(2)\Gamma(\f52)}{ \Gamma(1)\Gamma(\f92)} \f1{A^3} \times 2.0963 \leq 1.2644 \times 10^{-21}.
\end{align*}  
 
Finally, putting the three cases together completes the proof of Theorem \ref{thm5.3}.
\newpage

\appendix

\section{Some useful results about Jacobi polynomials}

In this appendix, we collect some formulas and properties of the Jacobi polynomials that are needed in the paper.

	\begin{enumerate}[\rm (i)]
		\item  \cite[(4.22.2)]{Sz}  For a positive integer $1\leq \ell\leq n$ and $\be\in\RR$, 
		\begin{equation} \label{1-2} P_n^{(-\ell,\be)} (x) =\f { \binom{n+\be} {\ell}}{\binom{n} {\ell}} \Bl (\f {x-1}2\Br)^{\ell} P_{n-\ell}^{(\ell,\be)} (x).\end{equation}
		\item \cite[(4.21.7)]{Sz}  For $\al,\be \in\RR$, 
		\begin{equation}\label{1-3} \f d{dx} P_n^{(\al,\be)}(x) =\f 12 (n+\al+\be+1) P_{n-1}^{(\al+1, \be+1)}(x).\end{equation}
	\item \cite[(4.1.3)]{Sz}	\begin{equation}\label{1-4}		P_n^{(\al,\be)}(-x) =(-1)^n P_n^{(\be, \al)} (x). \end{equation}
	Using  \eqref{1-2} and \eqref{1-4}, we obtain that 
		\begin{equation}\label{1-7}
			P_n^{(-1,-1)}(x) =-\f {1-x^2}{4} P_{n-2}^{(1,1)}(x),
		\end{equation}
		which,  by \eqref{1-3}, also implies  
		\begin{equation}\label{1-8}
			\int P_n(x)\, dx =\f 2{n} P_{n+1}^{(-1,-1)}(x)=-\f {1-x^2} {2n} P_{n-1} ^{(1,1)}(x).
		\end{equation}

		\item \cite[Page 295]{Hob}
		 For $\t\in(0,\f\pi2),$
	 	\begin{align*} 
	 	P_n (\cos\ta)=\f{2}{\pi^\f12}\f{\Gamma(n+1)}{\Gamma(n+\f32)} \f{\cos(N\t-\f\pi 4)}{(2\sin \t)^\f12}  + p_{n,1}(\cos \t),
	 \end{align*} where $$|p_{n,1}(\cos \t)|\leq\f{4}{\pi^\f12}\f{\Gamma(n+1)}{\Gamma(n+\f32)}\f1{2(2n+3)}\f1{(2\sin \t)^\f32}.$$  Applying the Gautschi's inequality: for $s\in(0,1)$, $x^{1-s}<\f{\Gamma(x+1)}{\Gamma(x+s)}< (x+1)^{1-s}$ with $x=n+\f12$ and $s=\f12$, we have
			\begin{align}\label{propA.4}
				 P_n (\cos\ta) \leq \f{2}{\pi^{\f12}}  \f{1}{n^{\f12}} \f{1}{(2\sin\t)^{\f12}} \cos\big(N\t-\f{\pi}{4}\big)+\f{1}{(2\pi)^{\f12}(\f2\pi)^{\f32}}\f{1}{(n\t)^{\f32}}.
			\end{align} 
	 
		\item \mbox{\cite[(7.3.8)]{Sz}}
		For $\t\in (0,\f \pi2]$, we have	\begin{equation}\label{Legendre estimates}		|P_n  (\cos\ta)|\leq \f {1}{n^{\f12} \ta^{ \f12}},		\end{equation} 	
		
		\item By \mbox{\cite[Page 213]{ADM} and \cite{Bara}}, 
		we can estimate that
		\begin{equation}\label{C.3}			|P_{n-1}^{(1,1)}(\cos t)|\leq \f { \pi^{\f32}\times 2.821}{{(n-1)}^{\f12} t^{\f32}}.
		\end{equation}

		\item \cite[p. 302]{As}  For $\ld\neq0$,   
		\begin{equation} \label{1-5} C_n^\ld(\cos\t) = \sum_{k=0}^n \f{\Gamma(\l+k)\Gamma(\l+n-k)}{k!(n-k)! [\Gamma(\l)]^2} \cos[(n-2k)\t].\end{equation}
	\end{enumerate}

\section{A useful asymptotic estimation}

The following two useful asymptotic formulas  can be found in the book \cite [p. 24, (11.5),(11.6)]{Co}: 

\begin{prop} \label{lem-asym} 
	Let $\phi(t)$ be $v$ times continuously differentiable in $\a\leq t \leq \b$. 	Let $\phi(t)$ and its first  $v-1$ derivatives vanish when $t=\beta$. Then, if  $0<\lambda<1$, as $N\rightarrow \infty$,
	\begin{align*}  \int_\a^\b e^{iNt}   (t-\a)^{\l-1}\phi(t)  dt
	 =   \sum_{n=0}^{v-1} \f{\Ga(n+\l)}{ n! N^{n+\l}}  e^{iN\a +\f12\l\pi i+\f12 n \pi i}\phi^{(n)}(\a)   +  O(N^{-v}),
		\end{align*} where \begin{align*}
O(N^{-v})=\f{1}{N^v}\int_{\a}^{\b}|\phi^{(v)}(t)|(t-\a)^{\l-1}dt.
	\end{align*}
	\end{prop} 

\begin{prop} \label{lem-asym2}
	Let $\phi(t)$ be $v$ times continuously differentiable in $\a\leq t \leq \b$. 	Let $\phi(t)$ and its first  $v-1$ derivatives vanish when $t=\a$. Then, if  $0<\mu<1$, as $N\rightarrow \infty$,
	\begin{align*} 
		\int_\a^\b e^{iNt}   (\b-t)^{\mu-1}\phi(t)  dt 
		=   \sum_{n=0}^{v-1} \f{\Ga(n+\mu)}{ n! N^{n+\mu}}  e^{iN\b - \f12\mu\pi i+\f12 n \pi i}\phi^{(n)}(\b)   + O(N^{-v}).
	\end{align*}
\end{prop}

\section*{Acknowledgement}
 The authors are grateful to Professor Feng Dai, Professor Yuan Xu for helpful discussions. They also express their deep gratitude to an anonymous referee for giving many helpful comments and constructive suggestions that led to an improved presentation of this paper. The authors are grateful to Elena Berdysheva and Joseph Zelezniak for pointing out an issue in the statement of Theorem~1.7 in the published version. The first author is supported partially by the Research Grants Council
 of Hong Kong [Project \# CityU 21207019] and by the City University of Hong
 Kong [Project \# CityU 7200608]. The second author is supported partially by the NSERC Canada	under grant RGPIN-2020-03909.

\end{document}